\documentclass{amsart}

\newenvironment{thmenum}{

\begin{enumerate}}
{\end{enumerate}}

\usepackage[dvips]{graphicx}
\usepackage{epsfig, latexsym, amsfonts, amsthm, amsmath, pstricks, mathtools}
\newcommand{\powfield}[1]{( \hspace{-1.6pt} ( {#1} ) \hspace{-1.6pt} )}

\renewenvironment{proof}[1][]{\vskip-\lastskip\par\vskip6pt plus2pt
minus0pt\par%
\noindent\textit{Proof.}\enspace\ignorespaces}{\hfill$\Box$\par\vskip6pt
plus2pt minus0pt}

\numberwithin{equation}{section}

\newtheorem{theorem}[equation]{Theorem}

\newtheorem{lemma}[equation]{Lemma}

\newtheorem{proposition}[equation]{Proposition}

\newtheorem{corollary}[equation]{Corollary}

\theoremstyle{definition}
\newtheorem{definition}[equation]{Definition}

\newtheorem{remark}[equation]{Remark}

\newtheorem{remarks}[equation]{Remarks}

\newtheorem{question}[equation]{Question}

\newtheorem{conventions}[equation]{Conventions}

\newtheorem{notation}[equation]{Notation}

\newtheorem{plain}[equation]{}

\begin{document}
\title{Compact Totally Disconnected Moufang Buildings}
\author[T. Grundh\"ofer, L. Kramer, H. Van Maldeghem and R. M. Weiss]{Theo 
Grundh\"ofer, Linus Kramer, Hendrik Van Maldeghem and\\ Richard M. Weiss}
\address{T. Grundh\"ofer \\ Institute for Mathematics \\
         University of W\"urzburg \\
         Am Hubland \\
         97074 W\"urzburg, Germany}
\email{grundhoefer@mathematik.uni-wuerzburg.de}
\address{L. Kramer \\ Institute for Mathematics \\
         University of M\"unster \\
         Einsteinstr. 62 \\
         48149 M\"unster, Germany}
\email{linus.kramer@math.uni-muenster.de}
\address{H. Van Maldeghem \\ Department of Mathematics \\
	Ghent University \\
        Krijgslaan 281 \\
	9000 Ghent, Belgium}
\email{hvm@cage.ugent.be}
\address{R. M. Weiss \\ Department of Mathematics \\
         Tufts University \\
         503 Boston Avenue \\
         Medford, MA 02155, USA}
\email{rweiss@tufts.edu}
\thanks{This work was partially supported by the SFB 478 in M\"unster.
The second and fouth authors are partially supported by DFG research grant KR 1668/7-1.}

\begin{abstract}
Let $\Delta$ be a spherical building each of whose
irreducible components is infinite, has rank 
at least~$2$ and satisfies the
Moufang condition.
We show that $\Delta$ can be given the structure of a topological building that is
compact and totally disconnected precisely when 
$\Delta$ is the building at infinity of a locally finite 
affine building.
\end{abstract}

\maketitle

\section{Introduction}\label{000}
\noindent
Generalizing earlier work by Kolmogorov and Pontryagin \cite{kolmogoroff,pontrjagin},
Salzmann \cite{szm}, and Burns and Spat\-zier \cite{burns},
Knarr and two of the present authors showed 
that an infinite compact (locally) connected building without factors of rank~1
whose topological automorphism group
acts transitively on the chambers must be the building associated with the parabolic
subgroups of a semisimple Lie group;
see \cite{knarr}, \cite{knarr2}, \cite{kramer-diss}, \cite[Ch.~7]{kramerhabil}. In particular,
such a building is the spherical building at infinity of a unique
Riemannian symmetric space of noncompact type. (See Section~\ref{999} for
more details.)

The goal of this paper is to extend this result
to {\it totally disconnected} compact spherical buildings.
Our results are as follows. 
(All buildings in this paper are thick. 
By {\it Bruhat-Tits building}, we mean an irreducible affine building of dimension ${\it l}$ at
least~$2$ whose building at infinity satisfies the Moufang condition; see
Definition~\ref{thm1x}, Conventions~\ref{abc83} and Section~\ref{thm1p} below.)

\begin{theorem}\label{thm1}
Let $\Delta$ be an infinite irreducible 
spherical building of rank ${\it l}$ at least~$2$. Suppose, furthermore,
that $\Delta$ satisfies the Moufang condition 
if ${\it l}=2$. Then the following hold.
\begin{thmenum}
\item $\Delta$ can be given the structure of a totally disconnected compact building
if and only if it is the building at infinity of a locally finite
Bruhat-Tits building $X$ of dimension ${\it l}$.
\item If $\Delta$ is the building at infinity of a locally finite
affine building $X$ of dimension ${\it l}$, 
then $X$ is unique, the topology on $\Delta$ giving it the structure
of a compact building is unique, this topology 
is uniquely determined by the canonical ${\rm CAT}(0)$ metric on $X$,
the set of chambers of
$\Delta$ is totally disconnected and, in fact, homeomorphic to the Cantor set,
and the natural map from $\mathrm{Iso}(X)$ to $\mathrm{Aut}(\Delta)$, each 
equipped with the compact-open topology, is
an isomorphism of topological groups.
\end{thmenum}
\end{theorem}

Combining this with the aforementioned results, we have the following.

\begin{theorem}\label{thm2}
Let $\Delta$ be a compact spherical building each of whose irreducible factors
is infinite, has rank at least~$2$ and satisfies the Moufang condition.
Then the following hold.
\begin{thmenum}
\item $\Delta$ is the building at infinity
of a product of irreducible Riemannian symmetric spaces of noncompact type and rank
at least~$2$ and locally finite Bruhat-Tits buildings of dimension at least~$2$. 
\item The compact topology of $\Delta$ is 
unique up to conjugation by field automorphisms on the irreducible
Riemannian symmetric factors belonging to complex simple Lie groups.
\item If none of the Riemannian symmetric irreducible factors belongs to a 
complex simple Lie group, then the compact topology
of $\Delta$ is unique and every abstract automorphism is continuous.
\end{thmenum}
\end{theorem}

\begin{definition}\label{thm1x}
Let $\Delta$ an irreducible spherical building of rank at least~$2$. 
For each root $\alpha$ of each apartment, we denote by $U_\alpha$ the corresponding
root group. This is the subgroup of ${\rm Aut}(\Delta)$ consisting of all
elements acting trivially on all panels containing two chambers in $\alpha$.
The building $\Delta$ {\it satisfies the Moufang condition} (equivalently,
$\Delta$ {\it is Moufang}) if 
for each root $\alpha$ of each apartment, the root group $U_\alpha$
acts transitively on the set of all apartments of $\Delta$ that contain
$\alpha$. See \cite[11.2]{weiss1}.
\end{definition}

If $\Delta$ is an irreducible spherical building, then 
by \cite{tits-spherical}
and \cite[11.6]{weiss1}, $\Delta$ is automatically Moufang if its rank~${\it l}$
is at least~3. The hypothesis in Theorem~\ref{thm1} that 
$\Delta$ is Moufang if ${\it l}=2$ is, however, essential. 
In \cite{grund1}, for example, it is shown that compact totally
disconnected projective planes exist that do not have a continuous epimorphism
to a finite projective plane and hence cannot be the building at 
infinity of a locally finite affine building. 

The following question is presently open. In the (locally)
connected case the answer is in
affirmative in both cases by \cite{burns} and \cite{knarr,knarr2}.

\begin{question}
{\em Let $\Delta$ be an irreducible totally disconnected compact building
of rank at least~$2$. If the topological automorphism group
$\mathrm{Auttop}(\Delta)$ acts chamber transitively or even strongly transitively on
$\Delta$, is $\Delta$ necessarily Moufang?}
\end{question}

Bruhat-Tits buildings of dimension ${\it l}\ge2$ were classified by Bruhat and Tits
in \cite{bruh} and \cite{como}.
Tables describing the thirty-five families of locally finite Bruhat-Tits
buildings can be found
in \cite[Chapter~28]{weiss3}. Apart from the three families involving 
inseparable extensions in characteristic~2 and~3 in \cite[Table 28.4]{weiss3},
they are precisely the affine
buildings associated with absolutely simple algebraic groups of $k$-rank ${\it l}$,
where $k$ is a commutative local field in the sense of Definition~\ref{abc0} below. These
algebraic groups had been classified earlier in \cite{tits-local} (for arbitrary~${\it l}$).

\begin{conventions}\label{abc83}
All topological spaces in this paper are assumed to be Hausdorff, unless stated
otherwise. In a metric space $X$ we denote balls as
\[
 B_r(x)=\{y\in X;\ d(x,y)<r\}\text{ and }
 \bar B_r(x)=\{y\in X;\ d(x,y)\leq r\}.
\]
All buildings are assumed to be thick (as defined in \cite[1.6]{weiss1}).
When we say that $\Delta$ is the ``building at infinity of an affine building $X$'' in
Theorem~\ref{thm1} or 
elsewhere in this paper, we mean that $\Delta$ is the ``building at infinity
of $X$ with respect to the complete system of apartments'' as defined in 
\cite[8.5 and 8.25]{weiss3}.
\end{conventions}

\noindent
After preparing the ground in Sections~\ref{444} through \ref{111},
we prove Theorems~\ref{thm1} and~\ref{thm2} in Section~\ref{999}.
We also call attention to Theorem~\ref{TopHua} below.

\smallskip
\section{Locally compact groups}\label{444}

\noindent
To get started, we assemble a few well known facts about locally compact
groups. Recall that a locally compact space is $\sigma$-compact if it is a
countable union of compact subsets. Every second countable locally compact
space is $\sigma$-compact.

\begin{theorem}\label{OpenAction}
Let $G$ be a locally compact $\sigma$-compact group and let
$X$ be a locally compact space. Suppose that
$G\times X\to X$ is a continuous and transitive action.
Then the action induces a homeomorphism from the quotient space $G/H$ to $X$, where $H$
is the stabilizer of a point $x\in X$.
\end{theorem}

\begin{proof}
See \cite[10.10]{stroppel} or \cite[96.8]{blauesbuch}.
\end{proof}

\noindent
From \ref{OpenAction}, we obtain the ``Open Mapping Theorem'':

\begin{corollary}\label{open}
Let $f:G\to H$ be a continuous epimorphism of locally compact groups.
If $G$ is $\sigma$-compact, then $f$ is an open map. In particular,
$f$ is a topological isomorphism if $f$ is continuous and bijective.
\end{corollary}

\begin{proof}
The map $G\to G/\mathrm{ker}(f)$ is open and 
$G$ acts via $f$ on $H$ by left multiplication. By
\ref{OpenAction} this action induces a homeomorphism $G/\mathrm{ker}(f)\to H$.
\end{proof}

\noindent
We could not find a reference for the following standard fact.

\begin{theorem}\label{IsometriesOfProperSpace}
Let $(X,d)$ be a proper metric space, i.e., all closed balls in $X$ are compact.
Then the isometry group $\mathrm{Iso}(X)$, endowed with the compact-open topology,
is a locally compact second countable transformation group.
The stabilizer $H\subseteq \mathrm{Iso}(X)$ of any point $x\in X$ is compact.
If $H/H_n$ is finite for all $n\in \mathbb N$, where $H_n$ is the pointwise
stabilizer of the closed ball $\bar B_n(x)$, then $H$ is totally disconnected.
\end{theorem}

\begin{proof}
Since $X$ is a countable union of second-countable open sets
\cite[XI.4.1]{dug}, $X$ is second countable. 
Therefore the space $C(X,X)$ of all continuous maps from $X$ to $X$
is second countable in the compact-open topology; see \cite[XII.5.2]{dug}.
Thus we may use sequences to check continuity. By \cite[XII.7.2]{dug}, the
compact-open topology coincides with the topology of uniform convergence
on compact sets. We note also that $\mathrm{Iso}(X)$ is closed in $C(X,X)$.

We show first that $\mathrm{Iso}(X)$ is a topological group.
It is true in general that $C(X,X)$ is a topological semigroup (by \cite[XII.2.2]{dug}),
so we only have to check the continuity of inversion.
Suppose that $(g_n)_{n\in{\mathbb N}}$ converges in $\mathrm{Iso}(X)$ to $g$.
We have $d(g_n(z),g(z))=d(z,g_n^{-1}g(z))$ for each $n$. Substituting $y=g(z)$, this
becomes $d(g_n(z),g(z))=d(g^{-1}(y),g_n^{-1}(y))$.
Thus if $Y\subseteq X$ is compact, then $(g_n^{-1})_{n\in{\mathbb N}}$ converges on $Y$ 
uniformly to $g^{-1}$
because $(g_n)_{n\in{\mathbb N}}$ converges on $g^{-1}(Y)$ uniformly to $g$.

Now we show that some neighborhood of the identity of $\mathrm{Iso}(X)$ is compact.
Let $x\in X$ be any point and put $L=\{g\in \mathrm{Iso}(X);\ d(x,g(x))<1\}$. This is an
open neighborhood of the identity in $\mathrm{Iso}(X)$ because the evaluation map
$g\mapsto g(x)$ is continuous (by \cite[XII2.4]{dug}).
Since $L$ consists of isometries, $L$ is equicontinuous.
Let $y\in X$ be arbitrary and put $r=d(x,y)$. Then
$L(y)$ is contained in the compact ball $\bar B_{r+1}(x)$. Therefore $L(y)$ has
a compact closure. By Arzela-Ascoli (see \cite[XII.6.4]{dug}), $L$ has a compact closure
$\bar L$ in $\mathrm{Iso}(X)$.

Since $H\subseteq \bar L$ is closed, $H$ is compact.
If $H/H_n$ is finite for all $n$, then $H$ injects into the compact totally
disconnected space
$$\prod_{n\in\mathbb N}H/H_n$$ 
and is therefore totally disconnected.
\end{proof}

\smallskip
\section{Locally compact fields}

\noindent
In this section we assemble classification results for locally compact 
nondiscrete fields, skew fields
and octonion division algebras. 

\begin{proposition}\label{locfieldclass}
Let $F$ be an infinite nondiscrete locally compact field. Then the following hold.
\begin{thmenum}
\item If $\mathrm{char}(F)>0$, then $F$ is isomorphic to
the field of formal Laurent series
${\mathbb F}_q\powfield{t}$ for some $q$, where ${\mathbb F}_q$ is the field of $q$ elements.
\item If $\mathrm{char}(F)=0$, then either $F={\mathbb R}$ or ${\mathbb C}$
or $F$ is a finite extension of the $p$-adic field ${\mathbb Q}_p$ for some prime $p$.
\end{thmenum}
\end{proposition}

\begin{proof}
This holds by \cite[Chapter~1, Theorems~5 and~8]{weil}.
\end{proof}

\begin{proposition}\label{abc101}
If $F$ is an infinite nondiscrete locally compact field of characteristic $p>0$,
then $F/F^p$ is an extension of degree $p$.
\end{proposition}

\begin{proof}
If $F={\mathbb F}_q\powfield{t}$ for some power $q$ of $p$, then
$F^p={\mathbb F}_q\powfield{t^p}$. The claim holds, therefore, by 
\ref{locfieldclass}(i).
\end{proof}

\begin{proposition}\label{localskewfieldclass}
Let $F$ be an noncommutative nondiscrete locally compact skew field. Then 
$F$ is a cyclic algebra over a locally compact field, i.e., over one of the fields in 
Proposition~{\rm\ref{locfieldclass}} other than ${\mathbb C}$.
\end{proposition}

\begin{proof}
See \cite[58.11]{salzmann}.
\end{proof}

\begin{proposition}\label{abc65}
Let $F$ be a nondiscrete locally compact field or skew field. Then 
the locally compact topology on $F$ is unique except when 
$F={\mathbb C}$, in which case the locally compact topology 
is unique up to field automorphisms.
\end{proposition}

\begin{proof}
See the references at the bottom of \cite[p.\,332]{salzmann}.
\end{proof}

\begin{definition}\label{abc0}
We will say that $F$ is a {\it local field} if it is a field or a skew field
complete with respect to a discrete valuation whose residue field is finite.
\end{definition}

\begin{proposition}\label{abc0a}
Local fields, as defined in Definition~{\rm\ref{abc0}}, are precisely the fields and the skew
fields that appear in Propositions~{\rm\ref{locfieldclass}} and~{\rm\ref{localskewfieldclass}}
other than those whose center is ${\mathbb R}$ or ${\mathbb C}$.
Furthermore, the discrete valuation of a local field is unique.
\end{proposition}

\begin{proof}
By \cite[Theorem~21.6]{warner}, a local field is a locally compact 
field with respect to the topology induced by its valuation. By
\cite[Section~II, Corollary~2 to Proposition~3]{serre} and \cite[Corollary~2.2]{wadsworth},
the fields and the skew fields that appear in 
Propositions~\ref{locfieldclass} and~\ref{localskewfieldclass} 
(other than those with center ${\mathbb R}$ or ${\mathbb C}$) are, 
conversely, local fields. The uniqueness of the discrete valuation of a local
field holds by \cite[56.13]{salzmann} or \cite[23.15]{weiss3}.
\end{proof}

\begin{remark}\label{abc0b}
Let $F$ be a local skew field.
Since $F$ and $F^{\mathrm{opp}}$ have the same valuation, it follows by Propositions~\ref{abc65}
and~\ref{abc0a} that they have the same topology as locally compact fields.
\end{remark}

\begin{proposition}\label{localaltalgebras}
Let $F$ be a nonassociative alternative division algebra. Then
$F$ is an octonion division algebra over some 
infinite field $k$. If $F$ is
locally compact and nondiscrete, then $F$ is topologically isomorphic
to the real octonion division algebra ${\mathbb O}$. 
In particular, the locally compact topology on $F$ is unique
and $F$ is connected.
\end{proposition}

\begin{proof}
The first assertion was first proved in \cite{bruck} and \cite{kleinfeld};
see \cite[20.2-20.3]{TW} for another proof. Therefore $k$ is infinite
by, for example, \cite[9.9(v)]{TW}. Now suppose that $F$ is
locally compact and nondiscrete. 
Since the center $k$ of $F$ is closed, it is locally compact.
By \cite[Corollary~1 on p.\,465]{warner1}, $k$ is nondiscrete. Hence by \cite[28.4(i)]{weiss3} and 
\ref{locfieldclass}, $k$ must be ${\mathbb R}$ or ${\mathbb C}$.
By \cite[20.9]{TW}, $k$ cannot be ${\mathbb C}$. Thus $k={\mathbb R}$.
By \cite[9.4 and 9.7]{TW}, there is precisely one quaternion division
algebra ${\mathbb H}$ over ${\mathbb R}$. Let $N$
denote its reduced norm. Since $N({\mathbb H})=\{\alpha\in{\mathbb R};\ \alpha\ge0\}$, 
it follows from \cite[9.9(v)]{TW} that there is precisely
one octonion division algebra ${\mathbb O}$ over ${\mathbb R}$. (The
uniqueness of ${\mathbb H}$ and ${\mathbb O}$ was, in fact, first proved by
Frobenius.) Since
${\mathbb O}$ is finite-dimensional over its center,
an isomorphism from ${\mathbb O}$ to $F$
must, in fact, be a topological isomorphism (by \cite[58.6(i)]{salzmann}).
\end{proof}

\smallskip
\section{The rank~1 case: compact projective lines}

\noindent
Let $F$ be field, a skew field or an alternative division algebra.
Let $P=F\cup\{\infty\}$ and let $G$ denote the permutation group
on $P$ generated by the maps
$\tau_a:x\mapsto a+x$ for all $a\in F$
and the involution $i:x\mapsto -x^{-1}$ 
with the usual conventions that
$0^{-1}=\infty$, $\infty^{-1}=0$, $-\infty=\infty$
and $a+\infty=\infty$ for each $a\in F$. We call the pair $(G,P)$ 
the \emph{projective line over $F$} and denote it by $\hat F$.
We call $\hat F=(G,P)$ a \emph{compact projective line}
if $F$ is infinite, $P$ carries a compact topology and
$G$ acts as a group of homeomorphisms on $P$.

The following, combined with Proposition~\ref{localaltalgebras}, generalizes the
main result of \cite{grund}.

\begin{theorem}\label{TopHua}
Let $F$ be a field, a skew field or an alternative division ring 
and suppose that $\hat F=(G,P)$ is a compact projective line.
Let $F\subseteq P$ be endowed with the subset topology.
Then $F$ is a locally compact, non-discrete and $\sigma$-compact 
field, skew field or alternative division ring.
\end{theorem}

\begin{proof}
We divide the proof into several steps. We note that $P$ is topologically
just the one-point compactification of $F$.

\medskip\noindent
{\bf Step 1.} {\it The additive group $(F,+)$ is a locally compact topological group.
Moreover, the extended addition $F\times P\to P$ is continuous.}

\begin{proof}
See \cite[Theorem~9.4]{warner} and its proof.
\end{proof}

\noindent
{\bf Step 2.} {\it Both inversion and the map $(x,y)\mapsto xyx$ 
from $F\times F$ to $F$ are continuous.} 

\begin{proof}
We start with Hua's identity:
$$xyx=(x^{-1}-(x+y^{-1})^{-1})^{-1}-x$$
for all $x,y\in F$ such that $xy\neq 0,-1$.
(When $F$ is an octonion division algebra, then
by \cite[20.22]{TW}, the
subring generated by any two elements of $F$ is
associative. Thus Hua's identity holds also 
in this case.) The
map $x\mapsto -x$ is a homeomorphism of $F$, so it extends
continuously to its one-point compactification $P$.
Since the map $i:x\mapsto -x^{-1}$ is continuous,
it follows that the extended inversion $x\mapsto x^{-1}$ from $P$ to $P$
is also continuous. It remains only to check that with these
extended maps, Hua's identity holds also when $xy=0$ or $xy=-1$.
\end{proof}

\noindent
{\bf Step 3.} {\it The map $x\mapsto x^2$ is continuous on $P$, where $\infty^2=\infty$.}

\begin{proof}
On $F=P\backslash\{\infty\}$ 
we have $x^2=x\cdot 1\cdot x$ and on $P\setminus\{0\}$ we have $x^2=-i(i(x)^2)$.
The claim holds, therefore, by Step~2.
\end{proof}

\noindent
{\bf Step 4.} {\it $F$ is $\sigma$-compact and not discrete.}

\begin{proof}
The compact set $P$ is infinite and homogeneous. It follows that
$F$ is not discrete and there exists a countably infinite set
$C=\{c_n;\ n\geq 1\}\subseteq F\setminus\{0\}$ that has $0$ in its closure.
Let $K$ be a compact neighborhood of $0$.
The map $\phi:(x,y)\mapsto xyx$ is continuous. For $z\in F$ we have
by continuity $0\in\phi(\bar C\times \{z\})\subseteq\overline{\phi(C\times\{z\})}$.
Therefore, $c_nzc_n\in K$ for some $n$. It follows that
$$F=\bigcup_{n\in{\mathbb N}}c_n^{-1}Kc_n^{-1}.$$
Thus $F$ is $\sigma$-compact.
\end{proof}

From now on, we let $k$ denote the center of $F$.

\medskip\noindent
{\bf Step 5.} {\it If $\mathrm{char}(F)\ne2$, then the map
$x\mapsto2x$ from $F$ to itself 
is a homeomorphism and if $\mathrm{char}(F)=2$, 
then $k$ is a closed subset of $F$ and the Frobenius map from $k$ to $k^2$
is a homeomorphism.}

\begin{proof}
If $\mathrm{char}(F)\ne2$, then the
map $x\mapsto2x$ is continuous (by Step~1) and bijective, 
hence (by Step~4 and Corollary~\ref{open})
a homeomorphism.
Suppose that $\mathrm{char}(F)=2$. In this case,
$k=\{x\in F;\ (x+y)^2=x^2+y^2\text{ for all }y\in F\}$. By Steps~1 and ~3, it follows
that $k$ is closed in $F$. By Step~3,
the map $x\mapsto x^2$ from $k\cup\{\infty\}$ to $k^2\cup\{\infty\}$ is continuous. 
Since the set $k\cup\{\infty\}$ is the one-point compactification of 
the closed subset $k\subseteq F$, this map is a homeomorphism. Thus
the Frobenius map from $k$ to $k^2$ is a homeomorphism.
\end{proof}

By Steps~1 and~2, addition and inversion in $F$ are continuous maps.
To conclude that $F$ is a topological field (or skew field
or alternative division ring), it thus remains only to show that
multiplication is also continuous.

\medskip\noindent
{\bf Step 6.} {\it If $F=k$ is a field, then multiplication is continuous.}

\begin{proof}
If $\mathrm{char}(k)\ne2$, then 
by Steps~1 and~3, the quantity
$$2xy=(x+y)^2-x^2-y^2$$ 
depends continuously on $x,y$, so by Step~5, 
the map $(x,y)\mapsto xy$ is continuous.
If $\mathrm{char}(k)=2$, then 
by Step~4, the inverse $s$ of the Frobenius map
$x\mapsto x^2$ is continuous, so by Steps~2 and~3, 
the map $(x,y)\mapsto s(xy^2x)=xy$ is continuous.
\end{proof}

\noindent
{\bf Step 7.} {\it If $F$ is a skew field, then multiplication is continuous.}

\begin{proof}
We put $\lambda_c(x)=cx$ for all $c,x\in F$ and $[a,b]=a^{-1}b^{-1}ab$
for all $a,b\in F^*:=F\backslash\{0\}$. Then
$$a^{-1}(b^{-1}(ab)x(ab)b^{-1})a^{-1}=[a,b]x=\lambda_{[a,b]}(x)$$ 
for all $a,b\in F^*$. By
Step~2, the maps $\lambda_c$ are thus continuous
for all multiplicative commutators $c\in[F^*,F^*]$.
By Step~6, we can assume that $F$ is not commutative. 
Let $a,b$ be non-commuting elements of $F$.
Then $a=(1-[a-1,b])([a,b]-[a-1,b])^{-1}$.
By 
Corollary~\ref{open}, the continuous automorphism
$$\lambda_{[a,b]-[a-1,b]}=\lambda_{[a,b]}-\lambda_{[a-1,b]}$$ 
of $(F,+)$ is a homeomorphism and therefore
$$\lambda_a=\lambda_{1-[a-1,b]}\circ\lambda^{-1}_{[a,b]-[a-1,b]}$$
is continuous on $F$. Since $a$ is an arbitrary non-central element
and every element in
$F$ is a sum of two non-central elements, we conclude that
the maps $\lambda_c$ are continuous for all $c\in F$.
Similarly, the maps $x\mapsto xc$ are continuous for all $c\in F$.
By \cite[Theorem~11.17]{warner}, it follows 
that the map $(x,y)\mapsto xy$ is continuous.
\end{proof}

\noindent
{\bf Step 8.} {\it If $F$ is a nonassociative alternative division algebra, 
then multiplication is continuous.}

\begin{proof}
By Proposition~\ref{localaltalgebras}, 
$F$ is an octonion division algebra and $k$ is infinite.
By \cite[20.8 and 20.22]{TW}, every subalgebra of $F$
generated by two elements is contained in an associative division ring. 
Let $j_y(x)=yxy$ for all $x,y\in F$. 
We claim that $k$ is closed in $F$. By Step~5, we can assume
that $\mathrm{char}(k)\ne 2$. We will show that in this case,
\begin{equation}\label{upp2}
k = \{ a\in F \mid j_a (j_x(y)) = j_x( j_a (y))
    \text{\ for all\ } x,y\in F\}.
\end{equation}
By Step~2, this will suffice to prove our claim. 
Let $a\ne 0$ belong to the set on the right-hand side of \eqref{upp2}. Then 
\begin{equation}\label{upp1}
a(xyx)a = x(aya)x
\end{equation}
for all $x,y \in F$. Setting
$y=a^{-1}$ in \eqref{upp1} yields $a(xa^{-1}x)a = xax$. Hence
$a^{-1}xa$ commutes with $x$ for all $x\in F$. 
Fix $x\in F\backslash\{k\}$ and 
let $B$ denote the subalgebra of $F$ generated by $x$.
By \cite[20.9]{TW}, $B$ is a subfield of $F$ and $B/k$ is 
a quadratic extension. By \cite[20.20, 20.21]{TW}, it follows
that the centralizer of $B$ in $F$ is $B$ itself.
Thus $a^{-1}Ba=B$. Hence
conjugation by $a$ induces a $k$-automorphism of order at most~2 on $B$
and therefore $a^2$ centralizes $x$. Since $x$ is arbitrary, it follows that $a^2\in k$.
Setting $y=1$ in \eqref{upp1}, we therefore have $ax^2 a = xa^2 x=x^2a^2$ for all $x \in F$. 
Thus $a$ commutes with $x^2$  and hence with 
$x=((x+1)^2-x^2-1)/2$ for all $x\in F$, so $a\in k$. Therefore \eqref{upp2} holds.
We conclude that $k$ is closed (in all characteristics) as claimed.

By Step~6, $k$ is locally compact 
and by Step~4, $k$ is $\sigma$-compact. We have
\begin{align*}
2xa&=\big((x+1)^2-x^2-1\big)a\\
&=j_{x+1}(a)-j_x(a)-a
\end{align*} 
for all $x\in k$ and $a\in F$. Thus
if $\mathrm{char}(k)\ne2$, then by Steps~1, 2 and~5, 
the map $(x,a)\mapsto xa$ from $k\times F$ to $F$
is continuous and therefore 
$F$ is a topological vector space over $k$.
Suppose that $\mathrm{char}(k)=2$ and let $s$ denote the inverse of the
Frobenius map from $k$ to $k^2$. Since
$xa=j_{s(x)}(a)$ for all $x\in k^2$ and all $a\in F$,
it follows from Steps~2 and~5 that 
$F$ is a topological vector space over $k^2$. By Proposition~\ref{locfieldclass}(i), 
$\dim_{k^2}k=2$ and hence $\dim_{k^2}F=16$ if $\mathrm{char}(k)=2$.
By Step~5, $k^2$ is homeomorphic to $k$ and hence also 
locally compact and $\sigma$-compact. We conclude that
$F$ is a locally compact topological vector space of finite dimension $m$
(equal to 8 or 16) over a locally compact
and $\sigma$-compact field $L$ (equal to $k$ or $k^2$)
in all characteristics. 

Let $v_1,\ldots,v_m$ be a basis for $F$. The map $L^m\to F$ that
sends $(a_1,\ldots,a_m)$ to $a_1v_1+\cdots+a_mv_m$ is a continuous
bijective homomorphism. Since $L$ is locally compact and $\sigma$-compact, so is $L^m$.
Hence this map is a homeomorphism by Corollary~\ref{open}.
We conclude that multiplication in $F$ is continuous (since it is bilinear over $L$).

\end{proof}

\noindent
With this last step, the proof of Theorem~\ref{TopHua} is complete.
\end{proof}

\noindent
\begin{remarks}\label{abc90} 
We note that another proof of Theorem~\ref{TopHua} can be derived from \cite[Satz~1.2]{petersson}.
If we add ``totally disconnected'' to the hypotheses of Theorem~\ref{TopHua}, then
by Proposition~ \ref{localaltalgebras},
we can add ``$F$ is associative'' to the conclusions. 
\end{remarks}

\smallskip
\section{Buildings}\label{thm1p}

\noindent
In this section, we briefly record some basic notions and facts for buildings
which can be found in 
\cite{abramenko}, 
\cite{brown}, \cite{ronan}, \cite{tits-spherical} and \cite{weiss1}.
We view buildings as simplicial complexes. 

\smallskip\noindent
\textbf{Simplicial complexes.}
Let $V$ be a set and $S$ a collection of finite subsets
of $V$. If $\bigcup S=V$ and if $S$ is closed under going down
(i.e., $a\subseteq b\in S$ implies $a\in S$), then the poset
$(S,\subseteq)$ is called a \emph{simplicial complex}.
More generally, any poset isomorphic to such a poset $S$
will be called a simplicial complex.
The \emph{join} $S*T$ of two simplicial complexes $S,T$ is
the product poset; it is again a simplicial complex.
The automorphism group of a simplicial complex is the group of all
its poset automorphisms.

\smallskip\noindent
\textbf{Coxeter groups and buildings.}
Let $(W,I)$ be a Coxeter system. Thus $W$ is a group with a (finite)
generating set $I$ consisting of involutions and a presentation
of the form $W=\langle I;\ (ij)^{{\rm ord}(ij)}=1\text{ for all }i,j\in I\rangle$,
where the relation $(ij)^{{\rm ord}(ij)}=1$ is to be ignored whenever 
${\rm ord}(ij)=\infty$. For a subset
$J\subseteq I$ we put $W_J=\langle J\rangle$. Then $(W_J,J)$ is again a Coxeter
system. The poset $\Sigma=\bigcup\{W/W_J;\ J\subseteq I\}$, ordered by
reversed inclusion, is a simplicial complex, the \emph{Coxeter complex}
$\Sigma=\Sigma(W,I)$. The \emph{type} of a simplex $wW_J$ is
$t(wW_J)=I\setminus J$. The type function may be viewed as a non-degenerate
simplicial epimorphism from $\Sigma$ to the power set $2^I$ of $I$
viewed as a simplicial complex. We refer to \cite{humphreys} for more
details on Coxeter groups.

A (thick) \emph{building} $\Delta$ is a simplicial complex together with
a collection of subcomplexes called apartments which are
isomorphic to a fixed Coxeter complex $\Sigma$. The apartments
have to satisfy the following two conditions.
\begin{thmenum}
\item[{\bf(B1)}] For any two simplices $a,b\in\Delta$, there is an apartment $A$
containing $a,b$.
\item[{\bf(B2)}] If $A,A'\subseteq\Delta$ are apartments containing the simplices $a,b$, 
then there is a (type preserving) isomorphism $A\to A'$ fixing $a$ and $b$.
\item[{\bf(B3)}] Every non-maximal simplex is contained in at least three maximal
simplices.
\end{thmenum}
The property (B3) is called {\it thickness}. Sometimes this axiom is omitted, but for our purposes
it is convenient to assume that buildings are thick. Non-thick buildings can in a certain
way be reduced to joins of thick buildings and spheres.
An automorphism of a building is just a simplicial automorphism. 

The join of two buildings
is again a building. Conversely, a building decomposes as a join
if its Coxeter group is decomposable, i.e., if $I\subseteq W$
decomposes into two subsets which centralize each other;
see \cite[3.10]{ronan}. A building is called {\it irreducible} if
its Coxeter group is not decomposable. 

The type functions of the apartments are
pairwise compatible and extend to a non-degenerate simplicial epimorphism
$t\colon\Delta\to 2^I$. The cardinality of $I$ is the \emph{rank} of
the building, so 
$$\mathrm{rank}(\Delta)=\dim(\Delta)+1.$$ 
The automorphisms of a building that preserve the type function 
are called \emph{special automorphisms}.
A building of rank~1 is just a set (of cardinality at least~3 since we are
assuming buildings to be thick), where the apartments are the two-element subsets.

\smallskip\noindent
\textbf{Chambers and galleries.}
The maximal simplices in a building are called \emph{chambers}.
The \emph{chamber graph} of $\Delta$ is the undirected graph
whose vertices are the chambers of $\Delta$. Two chambers are
\emph{adjacent} if they have a codimension $1$ simplex in common.
A \emph{gallery} is a simplicial path in the chamber graph, and a
nonstammering gallery is a path where consecutive chambers are
always distinct. The chamber graph of a building is always
connected. A \emph{minimal gallery} is a shortest path in the
chamber graph.

\smallskip\noindent
\textbf{The distance function.}
Let $\mathbf i=(i_1,\ldots,i_m)$ be a sequence in $I^m$.
A gallery of type $\mathbf i$ is a gallery
$(c_0,c_1,\ldots,c_m)$ where $c_{k-1}\cap c_k$
contains a simplex of type $I\setminus\{i_k\}$.
If $(c_0,c_1,\ldots,c_m)$ is a minimal gallery of type $\mathbf i$, then the product
$w=i_1\cdots i_m\in W$ is independent of the chosen gallery (i.e., the types of all
minimal galleries yield the same element $w\in W$). The $W$-valued distance
$\delta$ is defined by $\delta(c_0,c_m)=w$. Thus 
$$\delta(a,b)=\delta(b,a)^{-1}$$
for all chambers $a,b\in \Delta$ and $\delta(a,b)=1$ if and only if $a=b$. 

\smallskip\noindent
\textbf{Residues and panels.}
Let $a\in\Delta$ be a simplex of type $J$. The \emph{residue} of $a$ is the poset
$\mathrm{Res}(a)$ consisting
of all simplices containing $a$; this poset is again a
building, whose Coxeter complex is modeled on $W_{I\setminus J}$
\cite[3.12]{tits-spherical}.
If $a$ is a simplex of codimension $1$ and type $I\setminus\{j\}$,
then ${\rm Res}(a)$ is called a \emph{$j$-panel}.

\smallskip\noindent
\textbf{Spherical buildings.}
A Coxeter complex $\Sigma$ is called \emph{spherical} if it is
finite and a building is called \emph{spherical} if its apartments are finite.

\smallskip
\section{Compact buildings}
\noindent
\begin{definition}\label{abc91}
Let $\Delta$ be a spherical building with Coxeter system $(W,I)$.
We call $\Delta$ a \emph{compact spherical building} if for
each $i$ the set
$V_i$ of vertices of type $i$ carries a compact topology such that 
the set of chambers is closed in the product 
$\prod_{i\in I}V_i$, where we view a chamber as an $I$-tuple of vertices.
It is easy to see that this agrees with the definition of a compact
spherical building given in \cite[1.1]{burns}.
In the rank~2 case it also agrees with the definition of a
compact generalized polygon; see \cite{knarr}, \cite{vanmaldeghem}, \cite{kramer-diss}.
We say that $\Delta$ is connected, totally disconnected, etc., if the set of 
chambers is connected, totally disconnected, etc.
\end{definition}
Our first observation is that compact spherical buildings arise from
transitive actions of compact groups on spherical buildings.
\begin{lemma}\label{gup1}
Let $\Delta$ be a spherical building and suppose  that a group $K$ acts as a
chamber transitive automorphism group. Let $\{v_i;\ i\in I\}$
denote the vertices of a fixed chamber $c$. If $K$ is a compact group
and if the stabilizers $K_{v_i}$ are closed in $K$, then
$\Delta$ is a compact building if
we endow the set of $i$-vertices with the compact topology 
of $K/K_i$. 
If $K$ is totally disconnected, then $\Delta$ is totally disconnected.
\end{lemma}

\begin{proof}
The group $K_c=\bigcap\{K_{v_i};\ i\in I\}$ is closed
and the natural map
$$K/K_c\to\prod_{i\in I}K/K_{v_i}$$
is continuous onto the set of chambers. Since $K/K_c$ is compact,
the set of chambers is compact and therefore closed.
If $K$ is totally disconnected, then each coset space
$K/K_v$ is also totally disconnected; see \cite[7.11]{hewitt}.
\end{proof}

\begin{proposition}\label{pop3}
Let $G$ be a noncompact simple real or complex Lie group and let
$\Delta$ denote the spherical building associated to the canonical
Tits system of $G$; see \cite[p.\,68]{gwarner}. Then $\Delta$
carries in a natural way a compact topology which turns $\Delta$
into a compact building. Let $K\subseteq G$
be a maximal compact subgroup and let $X=G/K$ denote the associated
Riemannian symmetric space of noncompact type. Then $\Delta$ can be
identified with the spherical building at infinity of $X$.
The compact topology on $\Delta$ coincides with the topology induced
on $\Delta$ by the cone topology on $\partial X$. See \cite[II.8]{bridson}
for the latter.
\end{proposition}
\begin{proof}
We refer to \cite{gwarner} and \cite{eberlein} for the following facts.
Let $G=KAU$ be an Iwasawa decomposition of $G$; see \cite[Section~IX.1]{helgason}.
Note that we call the nilpotent part $U$ instead of $N$ since $N$ has a
different meaning for Tits systems. Let $K_0={\rm Cen}_K(A)$. Then
$K_0$ is the reductive anisotropic kernel of $G$ and $B=K_0AU$ is a minimal
parabolic subgroup. Let $N={\rm Nor}_G(K_0A)$. Then $(B,N)$ constitutes
a Tits system for $G$, of rank $\dim A$.
The parabolics of this Tits system are closed subgroups. The group $K$ acts
transitively on the chambers of the corresponding spherical building $\Delta$,
since $G=KB$. Because the parabolics are closed in $G$, the vertex stabilizers
in $K$ are also closed. By Lemma~\ref{gup1}, $\Delta$ is a compact building and
$G$ acts continuously.
Let $x=K$ denote the unique fixed point of $K$ in $G/K$. Then $K$ acts
continuously on the compact space $X\cup \partial X$; see \cite[II.8.8]{bridson}.
Recall that each boundary
point in $\partial X$ is represented by a geodesic ray starting at $x$
\cite[II.8.2]{bridson}.
Let $\xi$ be such a geodesic ray starting at $x$. Since $K$ is compact, the
$K$-orbit of $\xi$ is homeomorphic to the coset $K/K_\xi$.
The $G$-stabilizer of $\xi(\infty)$ is a parabolic $P$ (see \cite[2.17]{eberlein})
and therefore we have a $K$-equivariant homeomorphism $G/P\cong K/K_\xi$.
\end{proof}
In a similar way, we have the following result for Bruhat-Tits buildings.
Recall that by `Bruhat-Tits building,' we mean an irreducible affine building whose
spherical building at infinity is Moufang.
\begin{proposition}
\label{abc97}
Let $X$ denote a locally finite Bruhat-Tits building. Then ${\rm Aut}(X)$ is a totally
disconnected locally compact group. The spherical building at infinity, $\Delta$, is in a
natural way a totally disconnected compact building on which ${\rm Aut}(X)$ acts
continuously.
The compact topology on $\Delta$ coincides with the topology induced
on $\Delta$ by the cone topology on $\partial X$ with respect to the
canonical ${\rm CAT}(0)$ metric on $X$.
\end{proposition}
\begin{proof}
We endow the geometric realization of $X$ with the unique metric which
extends the $W$-invariant euclidean metric on the apartments. In this way,
$X$ becomes a CAT(0) space; see \cite[II.10A.4]{bridson} and
\cite[3.2]{bruh}. Moreover, each closed ball in $X$ is contained in a finite
subcomplex and is therefore compact.
It follows from Theorem~\ref{IsometriesOfProperSpace} that the isometry group
$\mathrm{Iso}(X)$ is locally compact and the stabilizer of any vertex $x\in X$ is compact.
Moreover, the isometry group ${\rm Iso}(X)$ coincides with the combinatorial automorphism
group ${\rm Aut}(X)$; here note that because $X$ is thick, every isometry is a simplicial
automorphism.
For every $n\geq 1$, the closed ball $\bar B_n(x)$ is contained in a finite
metric simplicial subcomplex of $X$. Therefore $\mathrm{Iso}(X)$ induces a finite isometry group
on $\bar B_n(x)$. By Theorem~\ref{IsometriesOfProperSpace}, $\mathrm{Iso}(X)_x$ is totally disconnected.
We claim that $\mathrm{Iso}(X)$ is also totally disconnected.
If $Z\subseteq \mathrm{Iso}(X)$ is a connected subset, then $Z$ intersects every stabilizer
trivially. If $x$ is a simplicial vertex, then the orbit $Z.x$ is a connected set of vertices.
But simplicial vertices are isolated in $X$.
Therefore $Z.x=\{y\}$ consists of a single element, and so does $Z$.

Now let $x$ be a special vertex. By Proposition~\ref{abc98}(iv) below, the compact group
$K=\mathrm{Iso}(X)_x$
acts transitively on the chambers of the spherical building at infinity.
The rest of the proof is completely analogous to the previous proof.
The $K$-stabilizer of each geodesic ray $\xi$ starting in $x$ is closed.
By Lemma~\ref{gup1}, the spherical building at infinity is a totally disconnected
compact building.
\end{proof}
Suppose now that $\Delta$ is a compact building over $I$ and let
$C$ denote its set of chambers. For each sequence $\mathbf i=(i_1,\ldots,i_m)\in I^m$,
we have the set 
of all (possibly stammering) galleries of
type $\mathbf i$. This is clearly a closed, whence compact, subset of
$C^{m+1}$. The following observations are immediate consequences.
\begin{lemma}
Let $\Delta$ be a compact spherical building and $C$ its set of chambers.
Let $m\geq 1$.
\begin{thmenum}\label{BasicLemmaOnCompactBuildings}
 \item For each $\mathbf i\in I^m$, the set of all pairs of chambers
that can be joined by a gallery of type $\mathbf i$ is compact.
 \item For each $\mathbf i\in I^m$, the set of all chambers
that can be reached from a given chamber by a gallery of type $\mathbf i$ is compact.
 \item The set of all chambers in any residue of $\Delta$ is compact. In particular,
panels are compact.
 \item The set of all pairs of opposite chambers is open in $C\times C$.
 \item The set of all chambers opposite a given chamber is open in $C$.
\end{thmenum}
\end{lemma}
\begin{proof}
The sets given in (i) and (ii) are continuous images of compact sets of galleries
and therefore compact. Now (iii) is a special case of (ii), since the set
of chambers of a residue is precisely the set of all chambers of $\Delta$
that can be reached from a chamber in the residue via a (possibly stammering)
gallery whose type represents the longest word $w_0$ in the Coxeter group
of the residue by, for example, \cite[5.4]{weiss1}.
The complements of the sets given in
(iv) and (v) are finite unions of compact sets by (i) and (ii), and therefore closed.
\end{proof}
\begin{corollary}
\label{ResiduesAreCompact}
Every residue in a compact spherical building is in a natural way a
compact spherical building.
\end{corollary}
\begin{proof}
The proof is mainly a matter of careful bookkeeping.
Let $a\in\Delta$ be a simplex of type $J\subseteq I$. Its residue $\mathrm{Res}(a)$ is 
by definition the poset of all simplices containing $a$. 
For $i\in I\setminus J$, its set of $i$-vertices $V_i(\mathrm{Res}(a))$ consists
of all simplices of type $J\cup\{i\}$ containing $a$.
Its chamber set is $C(\mathrm{Res}(a))=\mathrm{Res}(a)\cap C$. 
This set $C(\mathrm{Res}(a))$ is compact by 
\ref{BasicLemmaOnCompactBuildings}(iii). Since it projects continuously
onto $V_i(\mathrm{Res}(a))$, the latter set is also compact.
Moreover,  $V_i(\mathrm{Res}(a))\subseteq \prod_{j\in J\cup\{i\}}V_j$ is Hausdorff.
Consider the natural injection $C(\mathrm{Res}(a))\to\prod_{i\in I\setminus J}V_i(\mathrm{Res}(a))$.
This map is continuous, so its image is compact and therefore closed
in the product $\prod_{i\in I\setminus J}V_i(\mathrm{Res}(a))$.
\end{proof}
\begin{corollary}\label{CompactProductDecomposition}
Let $\Delta_1$ and $\Delta_2$ be spherical buildings and let
$\Delta=\Delta_1*\Delta_2$ denote their join, i.e., the product poset.
\begin{thmenum}
 \item If $\Delta_1$ and $\Delta_2$ are compact buildings, then
$\Delta$ is in a natural way a compact building.
 \item If $\Delta$ is a compact building, then $\Delta_1$ and
$\Delta_2$ are compact buildings and $\Delta$ is as in {\rm(i)} isomorphic
to their join.
\end{thmenum}
\end{corollary}
\begin{proof}
Again, the proof is just a matter of bookkeeping.
The vertex set of $\Delta_1*\Delta_2$ is the disjoint union of the vertex
sets of $\Delta_1$ and $\Delta_2$. The chamber set of $\Delta$ is the
cartesian product of the chamber sets of $\Delta_1$ and $\Delta_2$. 
For (i) we note that we have compact Hausdorff topologies on the vertex
sets of $\Delta$ and the chamber set of $\Delta$ is a continuous image of a compact
set in their cartesian product. This proves (i). For (ii) we note that
$\Delta_1$ and $\Delta_2$ are residues in $\Delta$. By 
Corollary~\ref{ResiduesAreCompact}, both residues are compact buildings. It is
now easy to see that the various compact topologies on $\Delta_1*\Delta_2$
and $\Delta$ coincide, since there are natural continuous bijections between
them.
\end{proof}
\begin{lemma}
\label{GalleriesAreContinuous}
Let $\mathbf i=(i_1,\ldots,i_m)$ be a reduced expression of $i_1\cdots i_m=w\in W$,
i.e., $w$ cannot be expressed in a shorter way as a word in the generating set $I$.
The map which assigns to a pair of chambers $(c,d)$ at distance
$\delta(c,d)=w$ the unique gallery $(c_0=c,c_1,\ldots,c_m=d)$ of type $\mathbf i$
is continuous.
\end{lemma}
\begin{proof}
We use the fact that a map from a topological space $X$ to
a compact space $Y$ is continuous if and only if its graph is closed;
see \cite[XI.2.7]{dug}. In our situation $X=\{(a,b)\in C\times C\mid\delta(a,b)=w\}$
and $Y=C^{m+1}$. The graph of the map in question is the set of all
$(m+3)$-tuples $(c_0,c_m,c_0,\ldots,c_m)$, where $c_0,\ldots,c_m$
is a nonstammering gallery of type $\mathbf i$. Clearly, this set is closed in
$C^2\times C^{m+1}$.
\end{proof}
\begin{corollary}
\label{projContinuous}
Let $i\in I$ and $w\in W$ and suppose that ${\it l}(wi)<{\it l}(w)$, i.e., $w$ admits
a reduced expression that ends with the letter $i$. For any chamber $c\in C$,
let $p^i(c)$ denote the unique $i$-panel containing $c$.
Then the map that assigns to a pair of chambers $(c,d)$ at distance 
$\delta(c,d)=w$ the unique chamber $e=\mathrm{proj}_{p^i(d)}(c)$ with
$\delta(e,d)=i$ and $\delta(c,e)=wi$ is continuous.
\end{corollary}

\begin{plain}\textbf{Schubert cells.}\label{pdf50}
Let $w\in W$ and let $c_0\in C$ be a fixed chamber. We call the
set $C_w(c_0)=\{c\in C;\ \delta(c_0,c)=w\}$ a {\em Schubert cell.}
If $w_0$ is the longest element in $W$, then $C_{w_0}(c_0)$ is sometimes
called the \emph{big cell}. The big cell is just the set of all chambers
opposite $c_0$. By Lemma~\ref{BasicLemmaOnCompactBuildings}(v), this
set is open.
\end{plain}
\begin{proposition}
\label{Coordinates}
Let $w\in W$ be an element of word length ${\it l}(w)=m\geq 1$. Then
the Schubert cell $C_w(c_0)$ is homeomorphic to a product of $m$ punctured
panels, where `punctured panel' means a panel with one chamber removed.
\end{proposition}
\begin{proof}
We proceed by induction on $m$. For each $i\in I$, we have 
$C_i(c_0)=p^i(c_0)\setminus\{c_0\}$.
Thus the claim holds for $m=1$.
Suppose now that $m>1$ and choose $i$ such that ${\it l}(wi)<{\it l}(w)$.
Let $w_0$ be as in \ref{pdf50} and let $j=w_0iw_0\in I$, put $v=ww_0$, so $w_0=v^{-1}w$, 
and fix $d\in C_v(c_0)$,
so $\delta(d,c)=v^{-1}w=w_0$ for all $c\in C_w(c_0)$.
Then the panels $p^i(a)$ and $p^j(d)$  are opposite for each 
$a\in C_{wi}(c_0)$ (by \cite[5.13]{weiss1}).
By \cite[5.14(i)]{weiss1},
$\mathrm{proj}_{p^i(a)}$ restricted to $p^j(d)$ and
$\mathrm{proj}_{p^j(d)}$ restricted to $p^i(a)$ are inverses of 
each other for each $a\in C_{wi}(c_0)$. Thus 
the maps $C_{wi}(c_0)\times C_j(d) \to C_w(c_0)$
sending $(a,b)$ to $c=\mathrm{proj}_{p^i(a)}(b)$ and 
$C_w(c_0)\to C_{wi}(c_0)\times C_j(d)$ 
sending $c$ to $(a,b)=(\mathrm{proj}_{p^i(c)}(c_0),\mathrm{proj}_{p^j(d)}(c))$
are inverses of each other. By Corollary~\ref{projContinuous}, both these maps are continuous.
By induction, therefore, $C_{wi}(c_0)$ is homeomorphic to a product of punctured panels.
\end{proof}
\begin{corollary}\label{pop1}
The topology on $\Delta$ is uniquely determined by its restriction
to the panels of $\Delta$.
\end{corollary}
\begin{proof}
By Proposition~\ref{Coordinates}, the topology on the panels determines
the topologies on the Schubert cells. In particular, it determines
the topologies on the big cells. Since these are open,
this determines the topology on $C$. The vertex sets $V_i$ are topologically
quotients of $C$ and thus the topology on $V_i$ is also determined
by the topologies of the panels.
\end{proof}
\begin{proposition}\label{pop2}
If $\Delta$ is an irreducible compact spherical building of rank
at least~$2$, then either $C$ is connected, or
$C$ is totally disconnected, and, in fact, homeomorphic
to the Cantor set.
\end{proposition}
\begin{proof}
Suppose that $C$ is not connected. Then there is some panel $R$ which is
not connected. By \cite[2.2.3]{kramer-diss}, $R$ is totally disconnected
and without isolated points.
If $S$ is any other panel in $\Delta$, then we find a sequence
of panels between $R$ and $S$ such that consecutive panels are
contained in irreducible rank~2 residues. Therefore, $S$
is also totally disconnected without isolated points by \cite[2.2.3]{kramer-diss}.
It follows that the big cells in $C$ are totally disconnected and
without isolated points,
and thus $C$ itself is totally disconnected.
Moreover, no point in $C$ is isolated, so the result follows
from \cite[2-98]{hocking}.
\end{proof}
We remark that the vertex sets $V_i$ are also either connected or totally disconnected,
depending on $C$. This can be proved in a similar way as in \cite[2.5.2]{kramer-diss}.
\begin{proposition}\label{CompactBuildingsMetrizable}
Let $\Delta$ be a compact building without factors of rank~$1$.
Then the topology on the vertices and chambers is separable
and metrizable.
\end{proposition}
\begin{proof}
In view of Corollary~\ref{CompactProductDecomposition} it suffices to consider the
case when $\Delta$ is irreducible and of rank $m\geq 2$. For
$m=2$, this is proved in \cite[Thm.~1.5]{knarr}. Suppose now that $m\geq 3$.
Every panel of $\Delta$ is then contained in some irreducible residue
of rank~2 and therefore second countable by \cite[Thm.~1.5]{knarr}.
It follows from
Proposition~\ref{Coordinates} that every big cell in $C$ is second countable. Since $C$ can
be covered by $|W|$ big cells (where $|W|$ here denotes the cardinality of $W$), 
and since the big cells are open by
Lemma~\ref{BasicLemmaOnCompactBuildings}(v), $C$ itself is second countable.
It follows that $C$ is separable and metrizable \cite[VIII.7.3, XI.4.1]{dug},
\cite[4.2.8]{eng}. It follows from \cite[4.4.15]{eng} that $V_i$ is also metrizable
and separable \cite[VIII7.2]{dug}.
\end{proof}

\begin{theorem}\label{AutIsMetrizable}
Let $\Delta$ be a compact building without factors of rank~$1$.
Then the group $\mathrm{Auttop}(\Delta)$ of all continuous automorphisms of
$\Delta$, endowed with the compact-open topology, is a
second countable locally compact metrizable group. The group of all
continuous special automorphisms is an open normal subgroup.
\end{theorem}
\begin{proof}
This is essentially \cite[Thm.~2.1]{burns}.
We first note that by Proposition~\ref{CompactBuildingsMetrizable}, $\Delta$ is a
metric building in the sense of \cite[p.\,12]{burns}. We also note that
a sufficiently small neighborhood of the identity of $\mathrm{Auttop}(\Delta)$
consists only of type-preserving automorphisms.
We note that the argument in \cite[pp.\,20--21]{burns} requires only
that each rank~1 residue is contained in some irreducible rank~2
residue. This is  precisely our assumption. Thus a sufficiently small
identity neighborhood of $\mathrm{Auttop}(\Delta)$ is compact and consists of
special automorphisms \cite[p.\,21]{burns}. Moreover, $\mathrm{Auttop}(\Delta)$
is second countable, because this is true for $\Delta$,
cp.~\cite[XII.5.2]{dug}, and therefore metrizable \cite[IX.9.2]{dug}.
\end{proof}

\begin{proposition}
Let $f:\Delta\to\Delta$ be an abstract automorphism of the compact
building $\Delta$. If, for each $i\in I$, there is some $i$-panel
$R_i\subseteq C$ such that $f$ is continuous on $R$, then
$f$ is a continuous automorphism of $\Delta$.
\end{proposition}
\begin{proof}
Suppose that $f$ is continuous on the panel $R$. Let $S$ be a
panel opposite to $R$. By \cite[5.14(i)]{weiss1} the maps
$\mathrm{proj}_R|_S$ and
$\mathrm{proj}_S|_R$ are inverses of each other.
Moreover, $f|_S=\mathrm{proj}_{f(S)}\circ f\circ\mathrm{proj}_R|_S$ is
continuous by Corollary~\ref{projContinuous}. It follows from
\cite[3.30]{tits-spherical} that $f$ is continuous on every panel of the same
type as $R$. Therefore $f$ is continuous on every panel of $\Delta$.
From Proposition~\ref{Coordinates} we see that $f$ is continuous
on every Schubert cell. In particular, $f$ is continuous on the big cells.
Since these are open and cover $C$, $f$ is continuous on $C$.
Finally, we note that $C\to V_i$ is an $f$-equivariant quotient map,
so $f$ is also continuous on each vertex set $V_i$.
\end{proof}

\begin{corollary}\label{ElationsAreContinuous}
Let $\Delta$ be a compact spherical building. Then every element of every root group
is continuous.
\end{corollary}
\begin{proof}
A root group fixes some panel of each type chamber-wise; see Definition~\ref{thm1x}.
Compare~\cite[5.1]{burns}.
\end{proof}

\begin{notation}\label{abc45}
Let $\Delta$ be an irreducible spherical building of rank at least~2
that satisfies the Moufang condition.
We denote by $G^\dagger$ the subgroup of $\mathrm{Aut}(\Delta)$ generated
by all the root groups of $\Delta$.
\end{notation}

\begin{corollary}\label{abc46}
Let $\Delta$ be an irreducible compact spherical building of rank at least~$2$
that satisfies the Moufang condition. Then $G^\dagger$ is contained in 
$\mathrm{Auttop}(\Delta)$. In particular, $\mathrm{Auttop}(\Delta)$ acts strongly
transitively on $\Delta$.
\end{corollary}

\begin{proof}
By \cite[11.12]{weiss1}, the group $G^\dagger$ defined in Notation~\ref{abc45}
acts strongly transitively on $\Delta$.
By Corollary~\ref{ElationsAreContinuous}, 
the group $G^\dagger$ is contained in $\mathrm{Auttop}(\Delta)$.
\end{proof}

\smallskip
\section{Moufang spherical buildings}\label{777}

\noindent
In this section, we gather the various results about Moufang spherical 
buildings we will need. See Appendix~B in \cite{weiss3} for a summary of the
classification of these buildings. 

Throughout this section, we let 
$\Delta$ be an irreducible spherical building of rank ${\it l}$ at least~2 
that satisfies the Moufang condition.

\begin{notation}\label{pop11}
Let $F$ be the defining field of $\Delta$ as defined in \cite[30.15]{weiss3}.
Note that although we call $F$ the defining {\it field} of $\Delta$, $F$ might be
a skew field or an octonion division algebra. Note, too, that although we call
$F$ {\it the} defining field of $\Delta$, it is not quite an invariant of $\Delta$.
By \cite[30.29]{weiss3}, the following hold.
\begin{thmenum}
\item If $\Delta$ is of type ${\sf A}_{\it l}$, then only the unordered pair $\{F,F^{\mathrm{opp}}\}$
is an invariant of $\Delta$.
\item If $\Delta$ is mixed as defined in \cite[30.24]{weiss3}, 
then $F$ is commutative, $\mathrm{char}(F)=p\le3$, 
there exists an extension $E$ of $F$ such that $E^p\subset F\subset E$
and the unordered pair $\{F,E\}$, but not $F$ alone, is an invariant of $\Delta$. Note
that if we replace $E$ by the isomorphic field $E^p$,
then $F^p\subset E\subset F$.
\item In every other case, $F$ is an invariant of 
$\Delta$.
\end{thmenum}
If $F$ is a local skew field, then so is $F^{\mathrm{opp}}$ (with the same valuation). 
If $F$, $E$ and $p$ are as in (ii), then (by Proposition~\ref{abc101})
$E$ equals either $F^p$ or $F$; in both cases, $E$ is isomorphic to $F$
and, in particular, $E$ is also a local field.
Thus in those few cases where there are really two defining fields, one
is a local field if and only if the other one is.
\end{notation}

\begin{theorem}\label{abc1}
Let $F$ be as in Notation~{\rm\ref{pop11}}. Then the following hold.
\begin{thmenum}
\item If $\Delta$ is the building at infinity of a locally finite affine building 
$X$, then $F$ is a local field and 
$X$ is one of the buildings in \cite[Tables~28.4--28.6]{weiss3}
where $F$ is called $K$.
\item If $F$ is a local field, then $\Delta$ is the building at infinity 
of a unique locally finite affine building $X$.
\end{thmenum}
\end{theorem}

\begin{proof}
If $\Delta$ is the building at infinity of a locally finite affine building 
$X$, then by \ref{abc83} and the remarks in the two paragraphs preceding \cite[28.1]{weiss3},
$F$ is a local field. Suppose, conversely, that $F$ is a local field.
By \cite[28.11(viii)]{weiss3}, $\Delta$
is not an exceptional quadrangle and therefore 
the relevant parameter system is  ``$\nu$-compatible'' in every case
(as defined in the places cited in the second column of \cite[Table~27.2]{weiss3}) 
simply because $F$ is complete with respect to its valuation $\nu$.
It follows from in \cite[27.2]{weiss3} that $\Delta$ is the building at infinity of 
an affine building $X$ with respect to some system of apartments ${\mathcal A}$ 
and from \cite[27.1(iv)]{weiss3} that the pair $(X,{\mathcal A})$ is unique
(since the valuation of $F$ is unique). 
It is the main result of \cite[Chapter~28]{weiss3} (see, in particular, 
\cite[28.11(viii), 28.14 and 28.33--28.38]{weiss3}) that if $F$ is a local field
and $\Delta$ is the building at infinity of an affine building $X$ with respect
to a system of apartments ${\mathcal A}$, then 
$(X,{\mathcal A})$ is one of the Bruhat-Tits pairs (as defined in \cite[13.1]{weiss3})
in \cite[Tables~28.4--28.6]{weiss3} (where $F$ is called $K$). It can be read
off from the last column of these tables that $X$ is, in every case, locally finite.

It is assumed in \cite[28.29]{weiss3} that ${\mathcal A}$ is complete, but
this assumption is not used in \cite[28.33--28.38]{weiss3}. Instead, it can
be observed by invoking
\cite[17.11, 19.28(ii), 23.13, 24.48, 25.27 and 28.15]{weiss3} that 
each Bruhat-Tits pair in these tables is, in fact,
complete as defined in \cite[17.1]{weiss3}. 
(When invoking \cite[23.13]{weiss3}, it needs to be observed that
by \cite[28.15]{weiss3}, $F_0$ is a closed subset of $F$
if $(F,F_0,\sigma)$ is an involutory set and when invoking \cite[25.27]{weiss3},
it needs to be observed that by Proposition~\ref{abc101},
$L$ is closed in $F$ if $\mathrm{char}(F)=3$ and $L$ is a subfield such that
$F^3\subset L\subset F$.)
\end{proof}

We now fix an irreducible rank~2 residue $\Gamma$ of $\Delta$.
By \cite[7.14, 7.15 and 11.8]{weiss1}, $\Gamma$ is a {\it Moufang $n$-gon} for some $n\ge3$
(as defined in \cite[4.2]{TW}). Thus, in particular, $\Gamma$ is a graph whose
vertices are the panel and whose edges are the chambers.

\begin{notation}\label{abc70}
Let $\Sigma$ be an apartment of $\Delta$ containing an apartment of $\Gamma$. 
By \cite[8.13(i)]{weiss1}, the intersection $\Sigma\cap\Gamma$ is a circuit of length $2n$ whose
vertices we can label with the integers modulo $2n$, so that $i-1$ and $i$ are adjacent
for each $i$. For each $i$, there exists 
by \cite[3.11 and 4.10]{weiss1} a unique root
$\alpha_i$ of $\Sigma$ such that $\alpha\cap\Gamma=(i,i+1,\ldots,i+n)$.
Let $U_i=U_{\alpha_i}$ for each $i$.
\end{notation}

\begin{proposition}\label{abc4}
Let $U_i$ for all $i$ be as in Notation~{\rm\ref{abc70}}. Then the following hold.
\begin{thmenum}
\item The group $\langle U_1,U_2,\ldots,U_n\rangle$ acts faithfully on $\Gamma$.
\item If $0\le j\le n-3$ and $1\le i\le n-j$, then the subgroup
$$U_iU_{i+1}\cdots U_{i+j}$$
is the pointwise stabilizer of its set of fixed points in $\Gamma$. 
\end{thmenum}
\end{proposition}

\begin{proof}
Assertion (i) holds by \cite[11.27]{weiss1}. Assertion (ii) 
holds, therefore, by \cite[5.2]{TW}.
\end{proof}

\begin{notation}\label{abc81}
By \cite[17.1--17.7]{TW} and \ref{abc4}(i), 
the group $U_+:=\langle U_1,U_2,\ldots,U_n\rangle$ can be identified with
one of the groups described in \cite[16.1--16.9]{TW}.
We denote the nine cases (as in \cite{TW}) ${\mathcal T}$, ${\mathcal Q}_{\mathcal I}$,
${\mathcal Q}_{\mathcal Q}$, ${\mathcal Q}_{\mathcal D}$, ${\mathcal Q}_{\mathcal P}$,
${\mathcal Q}_{\mathcal E}$, ${\mathcal Q}_{\mathcal F}$, ${\mathcal H}$
and ${\mathcal O}$. (We have $n=3$ in the first case, $n=4$ in the next
six cases, $n=6$ in Case~${\mathcal H}$ and $n=8$ in Case~${\mathcal O}$.)
In each case, $U_+$ is defined in terms of a parameter system $\Xi$. In Case~${\mathcal T}$, 
$\Xi$ is an alternative division ring $k$ (that is to say, a field,
a skew field or an octonion division algebra); in Case~${\mathcal Q}_{\mathcal I}$,
an involutory set $(k,k_0,\sigma)$; in Case~${\mathcal Q}_{\mathcal I}$, an 
anisotropic quadratic space $(k,L,q)$, in Case~${\mathcal Q}_{\mathcal D}$, an 
indifferent set $(k,k_0,L_0)$, in Case~${\mathcal Q}_{\mathcal P}$,
an anisotropic quadratic space $(k,k_0,\sigma,L,q)$, in Cases~${\mathcal Q}_{\mathcal E}$
and~${\mathcal Q}_{\mathcal F}$ a quadrangular algebra defined over a field $k$, 
in Case~${\mathcal H}$, an hexagonal system $(J,k,N,\#,T,\times,1)$ and in Case~${\mathcal O}$,
an octagonal set $(k,\sigma)$. (A quadrangular algebra is defined in \cite{weiss2}; 
references to all the other relevant definitions
can be found in \cite[16.1--16.9]{TW}.) By \cite[38.9]{TW}, we can assume
that in Case~${\mathcal Q}_{\mathcal D}$ the 
indifferent set $\Xi=(k,k_0,L_0)$ is proper as defined in \cite[38.8]{TW}.
\end{notation}

\begin{proposition}\label{abc80}
Let $k$ be as in Notation~{\rm\ref{abc81}} and let $F$ be the defining field of $\Delta$
as defined in Notation~{\rm\ref{pop11}}. Then one of the following holds.
\begin{thmenum}
\item $k=F$ or $F^{\mathrm{opp}}$.
\item $F$ is commutative and $F/k$ is a separable quadratic extension.
\item $F$ is a quaternion or octonion division algebra and $k$ is its center.
\item $F$ is commutative, ${\mathrm{char}}(F)=p=2$ or $3$ and either $F^p\subset k\subset F$
or $k^p\subset F\subset k$.
\end{thmenum}
\end{proposition}

\begin{proof}
This holds by \cite[30.14--30.15]{weiss3}.
\end{proof}

\begin{corollary}\label{abc80a}
Let $k$ be as in Notation~{\rm\ref{abc81}} and let $F$ be the defining field of $\Delta$.
Then the following hold.
\begin{thmenum}
\item $k$ is a local field if and only if $F$ is.
\item If $k$ is a local field, then either $k$ or $\{k,k^{\mathrm{opp}}\}$ is
an invariant of $\Delta$.
\end{thmenum}
\end{corollary}

\begin{proof}
By Propositions~\ref{localskewfieldclass} and~\ref{abc0a},
the center of a local field is a local field.
By Proposition~\ref{localaltalgebras}, there are no octonion division algebras over
a local field. Assertion (i) holds, therefore, by Propositions~\ref{abc0a} and~\ref{abc80}.
Assertion (ii) holds by Notation~\ref{pop11} applied to $\Gamma$
(since $\Gamma$ is an invariant of $\Delta$).
\end{proof}

\begin{proposition}\label{abc3}
Suppose that $\Delta$ is the building at infinity of a locally finite 
Bruhat-Tits building $X$, let $\Sigma$ be an apartment of $\Delta$,
let $\alpha$ be a root of $\Delta$ and let $\varphi$ be a valuation of 
the root datum of $\Delta$ based at $\Sigma$, which exists by the results
in the first two columns of \cite[Table~27.2]{weiss3} and is 
unique up to equipollence by \cite[3.41(iii)]{weiss3}.
Let the root group $U_\alpha$ be endowed with the metric $d_\alpha$ determined by
the map $\varphi_\alpha$ as described in \cite[17.3]{weiss3}.  
Then $U_\alpha$ is locally compact with respect to this metric.
\end{proposition}

\begin{proof}
We write $U_\alpha$ additively even though it might not be abelian and let 
$$U_{\alpha,k}=\{u\in U_\alpha;\ \varphi_\alpha(u)\ge k\}$$
for each integer $k$ (where $\varphi_\alpha(0)=\infty$).
By \cite[18.20]{weiss3}, $U_{\alpha,k+1}$ is a subgroup of finite
index in $U_{\alpha,k}$ for each $k$. To show that $U_\alpha$ 
is locally compact, it will suffice to show
that the open subgroup $U_{\alpha,0}$ is compact. Let $(u_i)_{i\in{\mathbb N}}$ be
an infinite sequence of elements in $U_{\alpha,0}$. We can choose
$w_0\in U_{\alpha,0}$ such that $-w_0+u_i\in U_{\alpha,1}$
for infinitely many $i$. We can then choose $w_1\in U_{\alpha,1}$ such
that $-w_1-w_0+u_i\in U_{\alpha,2}$ for infinitely many $i$.
We continue in this fashion and set $z_m=w_0+w_1+\cdots+w_m$
for each $m\ge0$. By \cite[17.9]{weiss3} and
Theorem~\ref{abc1}(i), $U_\alpha$ 
is complete with respect to the metric $d_\alpha$. It follows that the sequence 
$(z_m)_{m\in{\mathbb N}}$ has
a limit $z$ in $U_\alpha$. In fact, $z\in U_{\alpha,0}$ since $U_{\alpha,0}$ is closed.
By the choice of the elements $w_j$,
the element $z$ is an accumulation point of the sequence $(u_i)_{i\in{\mathbb N}}$. 
We conclude that every infinite sequence in $U_{\alpha,0}$ has an accumulation point
in $U_{\alpha,0}$, from which it follows that $U_{\alpha,0}$ is compact.
\end{proof}

\begin{proposition}\label{abc98}
Suppose that $\Delta$ is the building at infinity of a 
Bruhat-Tits building $X$,
let $G^\dagger$ be the subgroup of $\mathrm{Aut}(\Delta)$ 
defined in Notation~{\rm\ref{abc45}} and let $B$ be 
the stabilizer in $G^\dagger$ of a chamber $c$ of $\Delta$, i.e., $B$
is a minimal parabolic subgroup of $G^\dagger$.
Then the following hold.
\begin{thmenum}
\item The group $G^\dagger$ is canonically isomorphic to a subgroup of $\mathrm{Aut}(X)$.
\item The subgroup $B$ acts transitively 
on the set of all special vertices of a fixed type in $X$.
\item $G^\dagger=KB$, where $K$ is the stabilizer in $G^\dagger$ of a special vertex
of $X$. $($This is the {\it Iwasawa decomposition} of $G^\dagger$.$\,)$
\item The subgroup $K$ acts transitively on the chambers of $\Delta$.
\end{thmenum}
\end{proposition}

\begin{proof}
By \cite[12.31]{weiss3}, (i) holds and we can identify $B$ with a subgroup
of $\mathrm{Aut}(X)$. The chamber $c$ is a parallel class of sectors of $X$.
By \cite[11.12]{weiss1} and \cite[8.27]{weiss3},
$B$ acts transitively on the set of all apartments of $X$ containing $c$
(i.e.~containing a sector in the parallel class $c$).
By \cite[7.6]{weiss3}, every vertex of $X$ is contained in such an apartment.

Now let $A$ be an apartment of $X$ containing a sector in the parallel class $c$,
let $N$ denote the normalizer of $A$ in $G^\dagger$ and put $T=N\cap B$. 
By \cite[18.3(ii)]{weiss3}, $T$ contains all affine reflections of $A$.
By \cite[1.30]{weiss3}, it follows that $T$ acts transitively on the set
of special vertices of $A$ of any given type. Thus (ii) holds. This, in turn,
implies (iii) which then implies (iv). 
\end{proof}

\smallskip
\section{Compact totally disconnected buildings of rank~2}\label{111}

\noindent
In this section we prove the following result. It is the key step in the proof
of Theorem \ref{thm1}.

\begin{theorem}
\label{rank2classification}
Let $\Gamma$ be an infinite irreducible compact totally disconnected Moufang building
of rank ${\it l}=2$.
Let $G^\dagger$ be the group of automorphisms generated by all the root groups of $\Gamma$
and let $F$ be the defining field of $\Gamma$ as defined in Notation~{\rm\ref{pop11}}.
Then $F$ is a local field in the sense of Definition~{\rm\ref{abc0}}
and the topology on $\Gamma$ that makes it a compact building is unique.
\end{theorem}

The building $\Gamma$ in Theorem~\ref{rank2classification} is a graph whose edges are the chambers
and whose vertices are the panels. For each vertex $x$, we denote by $\Gamma_x$ the set of
vertices adjacent to $x$ and we identify
the set of chambers (i.e., edges) containing $x$ with $\Gamma_x$
via the map $\{x,y\}\mapsto y$.

We fix an apartment $\Sigma$ of $\Gamma$. 
Let $n$, 
the labeling of the vertices of $\Sigma$ 
and $U_i$ for $i\in{\mathbb Z}_{2n}$ be as in Notation~\ref{abc70}
and let $k$ and $\Xi$ be as in Notation~\ref{abc81}. 

We recall from Theorem~\ref{AutIsMetrizable} that the group $\mathrm{Auttop}(\Gamma)$ of topological
automorphisms of $\Gamma$ is locally compact, metrizable, second countable, and in
particular $\sigma$-compact. 

\begin{proposition}\label{abc6a}
The subgroup $U_i$ is a locally compact subgroup of $\mathrm{Auttop}(\Gamma)$ for all
$i\in[1,n]$. 
\end{proposition}

\begin{proof}
Let $i\in[1,n]$. By \ref{ElationsAreContinuous}, $U_i\subset\mathrm{Auttop}(\Gamma)$.
By Proposition~\ref{abc4}(ii) (with $j=0$) and Theorem~\ref{AutIsMetrizable}, 
therefore, $U_i$ is a closed subgroup of a locally compact group.
\end{proof}

From now on we use exponential notation to indicate the action of a group on 
a set; see \cite[11.9]{weiss1}.

\begin{proposition}\label{com1x}
The map $u\mapsto0^u$ is a homeomorphism from $U_1$ to $\Gamma_1\backslash\{2\}$
and the map $u\mapsto(n-1)^u$ is a homeomorphism from $U_n$ to $\Gamma_n\backslash\{n+1\}$.
\end{proposition}

\begin{proof}
The closed subgroup $U_1$ is $\sigma$-compact (by Theorem~\ref{AutIsMetrizable})
and the stabilizer in $U_1$ of the vertex $0$ is trivial (by \cite[3.7]{TW}). By
Theorem~\ref{OpenAction} and Proposition~\ref{abc6a}, it follows that
the map $u\mapsto0^u$ is a homeomorphism from $U_1$ to $\Gamma_1\backslash\{2\}$.
Thus the first assertion holds. The proof of the second assertion is virtually the same.
\end{proof}

\begin{notation}\label{com2a}
Let the cases 
${\mathcal T}$, ${\mathcal Q}_{\mathcal I}$,
${\mathcal Q}_{\mathcal Q}$, ${\mathcal Q}_{\mathcal D}$, ${\mathcal Q}_{\mathcal P}$,
${\mathcal Q}_{\mathcal E}$, ${\mathcal Q}_{\mathcal F}$, ${\mathcal H}$
and ${\mathcal O}$ and the parameter system $\Xi$ be as in Notation~\ref{abc81}.
We use $x_i$ to denote the various isomorphisms from the appropriate algebraic
structures to $U_i$ for all $i\in[1,n]$ that appear in \cite[16.1--16.9]{TW}.
In particular, the following hold.
\begin{thmenum}
\item $x_1$ is an isomorphism from $k_0$ to $U_1$ and $x_4$ is an isomorphism
from $L_0$ to $U_4$ in Case~${\mathcal Q}_{\mathcal D}$, where
$\Xi$ is the indifferent set $(k,k_0,L_0)$;
\item $x_1$ is the isomorphism
from the group called $S$ in \cite[16.6]{TW} to $U_1$ in Case~${\mathcal Q}_{\mathcal E}$; and
\item $x_1$ is the isomorphism from the group called $X_0\oplus K$ 
in \cite[16.7]{TW} (with $K$ here replaced by $k$) to $U_1$ in 
Case~${\mathcal Q}_{\mathcal F}$.
\end{thmenum}
In both (ii) and (iii), the map $t\mapsto x_1(0,t)$ is an isomorphism from the 
additive group of $k$ into $U_1$. 
\end{notation}

\begin{definition}\label{com2g}
In Case~${\mathcal Q}_{\mathcal D}$, where $\Xi=(k,k_0,L_0)$, we equip
$k_0$ with the unique topology that makes the isomorphism $x_1$ into
a homeomorphism from $k_0$ to $U_1$. 
In the remaining cases, we equip $k$ with the unique topology such that the following hold.
\begin{thmenum}
\item $x_1$ is a homeomorphism from $k$ to $U_1$
in Cases~${\mathcal T}$, ${\mathcal Q}_{\mathcal Q}$ and~${\mathcal O}$;
\item $x_n$ is a homeomorphism from $k$ to $U_n$ 
in Cases~${\mathcal Q}_{\mathcal I}$, ${\mathcal Q}_{\mathcal P}$ and~${\mathcal H}$;
\item the map $t\mapsto x_1(0,t)$ is a homeomorphism from $k$ to the subset
$x_1(0,k)$ of $U_1$ in Cases~${\mathcal Q}_{\mathcal E}$ and~${\mathcal Q}_{\mathcal F}$.
\end{thmenum}
\end{definition}

\begin{notation}\label{com1}
Let $\infty$, $P$ and $\langle t\rangle$ be as follows:
\begin{thmenum}
\item In Cases~${\mathcal Q}_{\mathcal I}$, ${\mathcal Q}_{\mathcal P}$ and~${\mathcal H}$,
let $\infty=n+1$, $P=\Gamma_n$ and $\langle t\rangle=(n-1)^{x_n(t)}$ for all $t\in k$.
\item In Cases~${\mathcal T}$,
${\mathcal Q}_{\mathcal Q}$ and ${\mathcal O}$,
let $\infty=2$, $P=\Gamma_1$ and $\langle t\rangle=0^{x_1(t)}$ for all $t\in k$.
\item In Case~${\mathcal Q}_{\mathcal D}$, where
$\Xi$ is an indifferent set $(k,k_0,L_0)$, let $\infty=2$, $P=\Gamma_1$ and
$\langle t\rangle=0^{x_1(t)}$ for all $t\in k_0$.
\item In Cases~${\mathcal Q}_{\mathcal E}$ and~${\mathcal Q}_{\mathcal F}$, let
$\infty=2$, $P=\Gamma_1$ and $\langle t\rangle=0^{x_1(0,t)}$ for all $t\in k$.
\end{thmenum}
\end{notation}

\begin{proposition}\label{com4}
Suppose that $\Gamma$ is neither in case
${\mathcal Q}_{\mathcal E}$ nor in case ${\mathcal Q}_{\mathcal F}$.
Let $P$ and $\langle t\rangle$ be as in Notation~{\rm\ref{com1}} and 
let $\tau$ be the map from $P$ to itself that interchanges $\infty$
and $\langle 0\rangle$ and maps $\langle t\rangle$ to $\langle-t^{-1}\rangle$
for all $t\in k^*$ $($respectively, $t\in k_0^*$ in case ${\mathcal Q}_{\mathcal D}$$\,)$.
Then $\tau$ is a homeomorphism.
\end{proposition}

\begin{proof} Suppose that we are in case ${\mathcal T}$, 
${\mathcal Q}_{\mathcal Q}$, ${\mathcal Q}_{\mathcal D}$ 
or ${\mathcal O}$. (In case ${\mathcal Q}_{\mathcal D}$,
where $\Xi=(k,k_0,L_0)$, ``$t\in k$'' is to be read as 
``$t\in k_0$'' everywhere in this proof.) By \cite[6.1]{TW},
there exists for each $t\in k^*$ a unique element $\mu(x_1(t))$
in the double coset $U_{n+1}^*x_1(t)U_{n+1}^*$ that maps $\Sigma$
to itself and induces on $\Sigma$ the reflection with fixed points
$1$ and $n+1$. The element $\mu(x_1(t))$ maps $P$ to itself
and interchanges $\infty$ and $\langle 0\rangle$ for all $t\in k^*$. Let 
$m=\mu(x_1(1))$. Then all the identities in \cite[32.5, 32.7, 32.8 and 32.12]{TW}
as well as
$$x_{n+1}(t)=x_1(t)^m$$
hold for all $t\in k$ (as explained in the introduction 
to \cite[Chapter~32]{TW}). In particular,
$$\mu(x_1(t))=x_{n+1}(t^{-1})x_1(t)x_{n+1}(t^{-1})$$
for all $t\in k^*$. Choose $t\in k^*$. Since both $m$ and $\mu(x_1(t))$
interchange $\langle 0\rangle$ and $\infty$ and $U_{n+1}$ fixes $\langle 0\rangle$,
we have
\begin{align*}
\langle t^{-1}\rangle^{mx_1(t)}=\langle0\rangle^{x_1(t^{-1})mx_1(t)}
&=\infty^{m^{-1}x_1(t^{-1})mx_1(t)}\cr
&=\infty^{x_1(t^{-1})^mx_1(t)}\cr
&=\infty^{x_{n+1}(t^{-1})x_1(t)}\cr
&=\infty^{\mu(x_1(t))x_{n+1}(t^{-1})^{-1}}\cr
&=\langle0\rangle^{x_{n+1}(t^{-1})^{-1}}=\langle0\rangle
\end{align*}
and therefore
$$\langle t^{-1}\rangle^m=\langle 0\rangle^{x_1(t)^{-1}}=\langle -t\rangle.$$
Thus $m$ induces the map $\tau$ on $P$. Hence (by 
\ref{ElationsAreContinuous}) $\tau$ is continuous. Since $\tau=\tau^{-1}$, 
it is, in fact, a homeomorphism.
The argument in the remaining cases is virtually identical, only with
$x_n(t)$ in place of $x_1(t)$, $x_n(t^{-1})$ in place of $x_1(t^{-1})$
and $x_0(t^{-1})$ in place of $x_{n+1}(t^{-1})$.
\end{proof}

\begin{proposition}\label{com6}
In Case~${\mathcal Q}_{\mathcal D}$, where $\Xi=(k,k_0,L_0)$, 
there exists a subfield $M$
of $k_0$ containing $k^2$ such that $x_1(M)$ is a 
closed subset of $U_1$, where $x_1$ is the isomorphism from $k_0$ to $U_1$
indicated in Notation~{\rm\ref{com2a}}.
\end{proposition}

\begin{proof}
Let $P$ and $\infty$ be as in \ref{com1},
let $R=\Gamma_4$, let $[\infty]=5$, let $[a]=3^{x_1(a)}$ for all $a\in L_0$ and let 
$\pi$ be the map from $R$ to $P$ that
sends $[\infty]$ to $\infty$ and $[a]$ to $\langle a\rangle$ for each $a\in L_0$.
Next let $m=\mu(x_1(1))$ and $r=\mu(x_4(1))$, where $\mu$ is as defined in \cite[6.1]{TW}.
By \cite[6.4(i) and 16.4]{TW}, we have
$$[x_1(1),x_4(a)]=x_2(a)x_3(a),$$
$x_2(a)=x_4(a)^m$ and $x_1(b)=x_3(b)^r$ for all $a\in L_0$ and all $b\in k_0$.
Hence
$$x_1(a)=\big(x_4(a)^{-m}x_4(a)^{x_1(1)}x_4(a)\big)^r$$
for all $a\in L_0$. It follows that the map
$x_4(a)\mapsto x_1(a)$ from $U_4$ to $U_1$ is continuous. By Proposition~\ref{com1x}, it
follows that the restriction $[a]\mapsto\langle a\rangle$ of 
the map $\pi$ to $R\backslash\{[\infty]\}$ is continuous. Exactly as in Proposition~\ref{com4}
(alternatively: by Proposition~\ref{com4} and \cite[35.18]{TW}), 
the map from $R$ to itself interchanging $[\infty]$ and $[0]$ and mapping
$[a]$ to $[a^{-1}]$ for all $a\in L_0^*$ is continuous. Thus a sequence
$([a_i])_{i\in{\mathbb N}}$ in $R\backslash\{[0],[\infty]\}$ converges to $[\infty]$ if and only
if the sequence $([a_i^{-1}])_{i\in{\mathbb N}}$ converges to $[0]$. Since the analogous
assertion holds in $P$, it follows that the map $\pi$ from $R$ to $P$ is continuous.
Since $R$ is compact (by Corollary~\ref{ResiduesAreCompact}), 
we conclude that its image is a compact subset of $P$. Thus $\pi(U_4)$
is a closed subset of $U_1$.
By Proposition~\ref{com1x} again, we conclude that $L_0$ is a closed subset of $k_0$.

Now let $M=\{t\in L_0;\ tL_0\subset L_0\}$. Then $M$ is a subring of $k_0$.
Since $k^2\subset M$ (by \cite[10.2]{TW}), it follows that $M$ is, in fact, a subfield.
Let $u\in L_0^*$.
The product $r^{-1}\mu(x_4(u^{-1}))$ fixes the vertex $2$ and (by \cite[32.8]{TW})
$$x_1(t)^{r^{-1}\mu(x_4(u^{-1}))}=x_3(t)^{\mu(x_4(u^{-1}))}=x_1(ut)$$
for all $t\in k$. 
Thus the map from $P$ to itself that fixes $\infty$ and maps
$\langle t\rangle$ to $\langle ut\rangle$ is continuous.
Now suppose that $(t_i)_{i\in{\mathbb N}}$ is a sequence of elements in $M$
that converges to an element $t\in L_0$. Then the sequence
$(ut_i)_{i\in{\mathbb N}}$ converges to $ut$. Since $ut_i\in L_0$ for all $i$,
it follows that $ut\in L_0$. Therefore $t\in M$. Hence $M$ is closed. 
\end{proof}

\begin{proposition}\label{com2}
In Cases~${\mathcal Q}_{\mathcal E}$ and~${\mathcal Q}_{\mathcal F}$,
$x_1(0,k)$ is a closed subset of $U_1$, where $x_1$ is as in
Notation~{\rm\ref{com2a}}.
\end{proposition}

\begin{proof}
By 16.10 and 16.11 of \cite{TW}, $x_1(0,k)=C_{U_1}(U_3)$. Since $U_1$ and $U_3$
are closed (by \ref{abc6a}), the claim follows.
\end{proof}

\begin{notation}\label{abc76}
Let $D$ be the group 
\begin{thmenum}
\item $U_n$ in Cases~${\mathcal Q}_{\mathcal I}$,
${\mathcal Q}_{\mathcal P}$ and ${\mathcal H}$;
\item $U_1$ in Cases~${\mathcal Q}_{\mathcal T}$, ${\mathcal Q}_{\mathcal Q}$~and 
${\mathcal O}$;
\item $x_1(M)$ in Case~${\mathcal Q}_{\mathcal D}$; and
\item $x_1(0,k)$ in Cases~${\mathcal Q}_{\mathcal E}$ and~${\mathcal Q}_{\mathcal F}$.
\end{thmenum}
\end{notation}

\begin{proposition}\label{abc77}
In Cases ${\mathcal Q}_{\mathcal E}$, ${\mathcal Q}_{\mathcal F}$
and~${\mathcal Q}_{\mathcal D}$, let $Q=\{\infty\}\cup0^D$, where
$D$ is as in Notation~{\rm\ref{abc76}}. In all other
cases, let $Q=P$, where $P$ is the panel of $\Gamma$ defined
in Notation~{\rm\ref{com1}}. Then $Q$ is a compact subset of $P$. 
\end{proposition}

\begin{proof}
By Corollary~\ref{ResiduesAreCompact}, the panel $P$ is compact. 
By Propositions~\ref{com6} and~\ref{com2}, therefore,
$Q$ is the one-point compactification of $Q\backslash\{\infty\}$
in Cases~${\mathcal Q}_{\mathcal E}$, ${\mathcal Q}_{\mathcal F}$
and ${\mathcal Q}_{\mathcal D}$.
\end{proof}

\begin{proposition}\label{com3}
The map $t\mapsto\langle t\rangle$ is a homeomorphism from $k$ 
$($respectively, $M$ in Case~${\mathcal Q}_{\mathcal D}$$\,)$ to $Q\backslash\{\infty\}$.
\end{proposition}

\begin{proof}
This holds by Proposition~\ref{com1x} and Definition~\ref{com2g}.
\end{proof}

\begin{proposition}\label{com3a}
$k$ $($respectively, $M$ in Case~${\mathcal Q}_{\mathcal D}$$\,)$ is totally disconnected,
in fact, homeomorphic to the Cantor set minus a point.
\end{proposition}

\begin{proof}
This holds by Propositions~\ref{pop2} and~\ref{com3}.
\end{proof}

\begin{proposition}\label{com4x}
Let $\omega$ be the map from $Q$ to itself that interchanges $\infty$ and $\langle 0\rangle$
and maps $\langle t\rangle$ to $\langle -t^{-1}\rangle$ for all $t\in k^*$
$($respectively, $t\in M^*$ in Case~${\mathcal Q}_{\mathcal D}$$\,)$. Then $\omega$
is a homeomorphism.
\end{proposition}

\begin{proof}
By Proposition~\ref{com4}, it suffices to consider Cases~${\mathcal Q}_{\mathcal E}$ and
${\mathcal Q}_{\mathcal F}$. In these two cases we can simply replace
$x_1(t)$ by $x_1(0,t)$, $x_1(t^{-1})$ by $x_1(0,t^{-1})$ 
and $x_{n+1}(t^{-1})$ by $x_{n+1}(0,t^{-1})$ everywhere in the proof of Proposition~\ref{com4}.
\end{proof}

\begin{proposition}\label{com30}
Let $E$ denote the permutation group on $Q$ generated by 
$\tau$ and the group induced by $D$ on $Q$ and let $k$ $($respectively, $M$ in 
Case~${\mathcal Q}_{\mathcal D}$$\,)$ be identified with $Q\backslash\{\infty\}$
via the homeomorphism $t\mapsto\langle t\rangle$, so $Q=k\cup\{\infty\}$ $($respectively,
$Q=M\cup\{\infty\}$$\,)$. 
Then $(E,Q)$ is the projective line $\hat k$ $($respectively, $\hat M$ in 
Case~${\mathcal Q}_{\mathcal D}$$\,)$.
\end{proposition}

\begin{proof}
This holds by Corollary~\ref{ElationsAreContinuous} and Proposition~\ref{com4x}.
\end{proof}

\begin{proposition}\label{com7}
$k$ is a local field as defined in Definition~{\rm\ref{abc0}}.
\end{proposition}

\begin{proof}
Suppose first that we are not in Case~${\mathcal Q}_{\mathcal D}$. 
By Proposition~\ref{abc77}, $Q$ is compact. By 
Theorem~\ref{TopHua} and Proposition~\ref{com30}, therefore, $k$ is a locally compact field. 
Hence by Propositions~\ref{localaltalgebras} and~\ref{com3a}, 
$k$ is a local field. Now suppose that we are 
in Case~${\mathcal Q}_{\mathcal D}$. By virtually the same arguments, we conclude that
$M$ is a local field. Since $M$ is a commutative field of characteristic~2,
it is the field of formal Laurent series over a finite field (by Proposition~\ref{locfieldclass}).
By Proposition~\ref{com6}, $k/M^2$ is a quadratic (and hence finite) extension.
Therefore $k$ is a local field also in this case.
\end{proof}

\begin{proposition}\label{com9}
The topology $T$ on the set of chambers of $\Gamma$ is unique.
\end{proposition}

\begin{proof}
Let $R$ be a panel of $\Gamma$. By Corollary~\ref{pop1}, it will suffice
to show that the restriction of $T$ to $R$ is unique. 
Let $\Sigma$, $n$ and $U_i$ for $i\in{\mathbb N}$ be as in \ref{abc70}
and let $k$ and $\Xi$ be as in \ref{abc81}.
We can assume that $\Sigma$ is chosen so that $R=\Gamma_1$ or $\Gamma_n$. 
Our goal, therefore, is to show that
the restriction of the topology $T$ to the panel $\Gamma_1$ and 
the restriction of $T$ to the panel $\Gamma_n$ are unique. 

By Theorem~\ref{abc1}(i), $\Gamma$ is in one of 
the cases~${\mathcal T}$, ${\mathcal Q}_{\mathcal I}$,
${\mathcal Q}_{\mathcal Q}$, 
${\mathcal Q}_{\mathcal P}$ or~ ${\mathcal H}$ 
and, as we saw in Proposition~\ref{com7}, $k$ is not octonion. 
Case~${\mathcal Q}_{\mathcal D}$
is ruled out by \cite[28.14]{weiss3}, Cases~${\mathcal Q}_{\mathcal E}$ 
and~${\mathcal Q}_{\mathcal F}$
are ruled out, as we observed in the proof of Theorem~\ref{abc1}, by
\cite[28.11(viii)]{weiss3} and Case~${\mathcal O}$ is
impossible since by Proposition~\ref{abc101}, there are no octagonal sets
$(k,\sigma)$, as defined in \cite[10.11]{TW}, with $k$ local.
By Corollary~\ref{abc80a}(ii), either
$k$ or $\{k,k^{\mathrm{opp}}\}$ is an invariant of $\Gamma$. 
By Proposition~\ref{abc0a} and Remark~\ref{abc0b}, the topology on $k$ that makes it
into a locally compact field is uniquely determined by $\Gamma$.
We call this topology $T_k$. 
Replacing $k$ by $k^{\mathrm{opp}}$ if necessary, we can assume by \cite[35.15]{TW} 
that $R=\Gamma_1$ if $\Gamma$ is in Case~${\mathcal T}$. 
It thus suffices to show that the restriction of $T$
to the panel $\Gamma_1$ is uniquely determined by $T_k$
in Case ${\mathcal T}$ and that the restriction of $T$ to 
$\Gamma_p$ is uniquely determined by $T_k$
for both $p=1$ and $p=n$ in every other case.

From now on, let $p=1$ in Cases~${\mathcal T}$ and~${\mathcal Q}_{\mathcal Q}$
and $p=n$ in Cases~${\mathcal Q}_{\mathcal I}$, 
${\mathcal Q}_{\mathcal P}$ and~${\mathcal H}$.
By Proposition~\ref{com3}, the restriction of $T$ to $\Gamma_p$ is uniquely determined by
$T_k$. We can assume from now on that we are not in Case~${\mathcal T}$.
Let $d$ be the element of the set $\{1,n\}$ different from $p$.
It remains only to show that the restriction of $T$ to $\Gamma_d$ is
uniquely determined by $T_k$.

Let $\nu$ be the valuation of $k$.
By the results in the third column of \cite[Table 27.2]{weiss3}, there is a valuation
$\varphi=(\varphi_i)$ of the root datum of $\Gamma$ (where $\varphi_i$ is 
a map from $U_i^*$ to ${\mathbb Z}$ for each root $\alpha_i$ of $\Sigma$) such that 
$\varphi_p(x_p(t))=\nu(t)$ for all $t\in k$.
Let $T_\varphi$ be the topology
on $U_d$ induced by the metric associated with $\varphi_i$ 
as described in Proposition~\ref{abc3}. 
By \cite[3.41(iii)]{weiss3}, $\varphi$ is unique up to equipollence.
This means that $\varphi_d$ is uniquely
determined by $T_k$ up to an additive constant. Hence
the topology $T_\varphi$ depends only on $T_k$.

Now let $T_d$ denote the locally compact topology on $U_d$ that 
we identified in Proposition~\ref{abc6a}.
Suppose we know that the identity map from $U_d$ to itself
is a continuous map from $(U_d,T_d)$ to $(U_d,T_\varphi)$. 
By Theorem~\ref{AutIsMetrizable} and Corollary~\ref{ElationsAreContinuous},
$U_d$ is $\sigma$-compact.
By Corollary~\ref{open}, it follows that $T_d=T_\varphi$. 
Hence, by the conclusion of the previous paragraph,
$T_d$ is uniquely determined by $T_k$. By Proposition~\ref{com1x}, it follows that 
the restriction of $T$ to $\Gamma_d$ is also uniquely determined by $T_k$.

Suppose now that we are in Case~${\mathcal Q}_{\mathcal Q}$, where $\Xi=(k,L,q)$
and $d=4$. By \cite[16.11]{TW}, we can assume that there is an element $1$ in 
$L$ such that $q(1)=1$. Let $m=\mu(x_1(1))$ and $r=\mu(x_4(1))$.
By \cite[6.4(i) and 16.3]{TW}, we have
$$[x_1(1),x_4(u)^{-1}]=x_2(u)x_3(q(u)),$$
$x_2(u)=x_4(u)^m$ and $x_3(t)^r=x_1(t)$ for all $u\in L$ and all $t\in k$.
Hence
$$x_1(q(u))=\big(x_4(u)^{-m}x_4(u)^{x_1(1)}x_4(u)^{-1}\big)^r$$
for all $u\in L$. Therefore the map $x_4(u)\mapsto x_1(q(u))$ is continuous.
By \cite[19.20]{weiss3}, we can assume that
$\varphi_4(x_4(u))=\nu(q(u))/\delta$ for all $u\in L$, where $\delta$ is either
1 or 2. It follows that the identity map from $U_4$ to itself is a continuous
map from $(U_4,T_4)$ to $(U_4,T_\varphi)$ and hence the restriction of $T$ to $\Gamma_4$
is uniquely determined by $T_k$.

Suppose next that we are in Case~${\mathcal Q}_{\mathcal P}$, 
where $\Xi=(k,k_0,\sigma,L,q)$ and $d=1$. Let $m=\mu(x_1(0,1))$
and $r=\mu(x_4(1))$. By \cite[6.4(i) and 16.5]{TW}, we have
$$[x_1(a,t),x_4(1)^{-1}]=x_2(t)x_3(a,t)$$
and hence
$$x_4(t)=\big(x_1(a,t)^{-1}x_1(a,t)^{x_1(1)^{-1}}x_1(a,t)^{-r^{-1}}\big)^{m^{-1}}$$
for all $x_1(a,t)\in U_1$. Thus the map $x_1(a,t)\mapsto x_4(t)$ is 
continuous. By \cite[24.33]{weiss3}, we can assume that
$\varphi_1(x_1(a,t))=\nu(t)/\delta$ for all $x_1(a,t)\in U_1$,
where $\delta=1$ or $2$. 
It follows that the identity map from $U_1$ to itself is a continuous map
from $(U_1,T_1)$ to $(U_1,T_\varphi)$ and hence the restriction of $T$ to $\Gamma_1$
is uniquely determined by $T_k$.

Suppose next that we are in Case~${\mathcal Q}_{\mathcal I}$, 
where $\Xi=(k,k_0,\sigma)$ and $d=1$. Replacing \cite[16.5]{TW} by \cite[16.2]{TW},
\cite[24.33]{weiss3} by \cite[23.4]{weiss3} and $x_1(a,t)$ by $x_1(t)$
in the previous paragraph, we obtain a proof 
that the identity map from $U_1$ to itself is a continuous map
from $(U_1,T_1)$ to $(U_1,T_\varphi)$ and hence the restriction of $T$ to $\Gamma_1$
is uniquely determined by $T_k$ also in this case.

Suppose, finally, that we are in Case~${\mathcal H}$, where
$\Xi=(J,k,N,\#,T,\times,1)$ and $d=1$. 
This time, we set $r=\mu(x_1(1))$ and $m=\mu(x_6(1))$. 
By \cite[16.8 and 29.17]{TW},
$$[x_1(a),x_6(1)^{-1}]=x_2(N(a))x_3(-a^\#)x_4(N(a))x_5(a)$$
and hence
$$x_4(N(a))x_5(a^\#)x_6(-N(a))=\big(x_1(a)^{-1}x_1(a)^{x_6(1)^{-1}}x_1(a)^{-r^{-1}}\big)^m$$
for all $a\in J$. By Proposition~\ref{abc4}(ii),
each of the groups $U_4U_5U_6$, $U_4U_5$ and $U_6$ is closed. By \cite[Theorem~6.23]{stroppel}, 
therefore, the map $x_4(w)x_5(b)x_6(t)\mapsto x_6(t)$ from $U_4U_5U_6$
to $U_6$ is continuous. We conclude that the map $x_1(a)\mapsto x_6(-N(a))$
is continuous. By \cite[15.23]{weiss3}, we can assume that
$\varphi_1(x_1(a))=\nu(N(a))/\delta$ for all $a\in J$, where $\delta=1$ or $3$.
It follows that the identity map from $U_1$ to itself is a continuous
map from  $(U_1,T_1)$ to $(U_1,T_\varphi)$ and hence the restriction of $T$ to $\Gamma_1$
is uniquely determined by $T_k$.
\end{proof}

This concludes the proof of Theorem~\ref{rank2classification}.

\smallskip
\section{The proofs of Theorem~\ref{thm1} and Theorem~\ref{thm2}}\label{999}

\noindent
{\em Proof of Theorem} \ref{thm1}.
We assume that $\Delta$ is an irreducible infinite compact totally disconnected
spherical Moufang building of rank ${\it l}\geq 2$. Then the panels of
$\Delta$ are totally disconnected. Let $\Gamma$ be an arbitrary
irreducible residue of $\Delta$ of rank~2.
By Corollary~\ref{ResiduesAreCompact}, $\Gamma$ is a compact
totally disconnected building and by \cite[11.8]{weiss1}, $\Gamma$ is Moufang.
If $\Gamma_1,\Gamma_2$ are two irreducible
rank~2 residues such that $\Gamma_1\cap\Gamma_2$ contains a panel, then by
\cite[34.5]{TW}, $\Gamma_1$ is finite if and only if $\Gamma_2$ is
finite. Since $\Delta$ is irreducible, it follows that all the irreducible 
rank~2 residues of $\Delta$ are finite if one of them is. Since $\Delta$ is 
infinite, we conclude that 
$\Gamma$ is infinite. Hence by
Theorem~\ref{rank2classification}, the totally disconnected
compact topology on $\Gamma$ is unique. Since $\Gamma$ is arbitrary, 
we conclude that the totally disconnected
compact topology on every
panel of $\Delta$ is unique. By Corollary~\ref{pop1}, therefore, the totally disconnected
compact topology on the set $C$ of chambers of $\Delta$ is unique and by
Proposition~\ref{pop2}, $C$ endowed with this topology is homeomorphic to the Cantor set.

Suppose now that $f$ is an abstract automorphism of $\Delta$. Then the
topology of $\Delta$, transformed by $f$, is another totally disconnected
compact topology
which turns $\Delta$ into a compact building. Therefore, $f$ fixes
the topology. In other words, $f$ is a homeomorphism. Thus, ${\rm Aut}(\Delta)=
{\rm Auttop}(\Delta)$. 

Next we note that the defining field $F$ of $\Delta$ is a local field
by Theorem~\ref{rank2classification}. By \cite[27.2 and 27.6]{weiss3}, 
therefore, there exists a unique locally
finite Bruhat-Tits building $X$ whose building at infinity is $\Delta$.
By Proposition~\ref{abc97}, the compact topology on $\Delta$ coincides with the 
topology induced by the cone topology on $\partial X$.

Finally, consider the natural homomorphism $\mathrm{Iso}(X)\to\mathrm{Aut}(\Delta)$.
By Proposition~\ref{abc97} again, 
the group $\mathrm{Iso}(X)$ acts continuously on $\Delta$. Since $\mathrm{Aut}(\Delta)$
carries the compact-open topology, the map $\mathrm{Iso}(X)\to\mathrm{Aut}(\Delta)$
is continuous. It is also injective (an isometry that fixes $\Delta$ pointwise fixes
every apartment in $X$ and is therefore the identity). By \cite[26.37 and 26.40-2]{weiss3}, 
this map is also surjective. Since $\mathrm{Iso}(X)$ is locally compact and $\sigma$-compact by
Theorem~\ref{IsometriesOfProperSpace}, it
follows from Corollary~\ref{open} that this map is open and therefore a homeomorphism.
This finishes the proof of Theorem~\ref{thm1}.
\qed

\medskip
Before we proceed with the proof of Theorem \ref{thm2}, we add some
general remarks on the compact connected case and its history.
Building on Pontryagin's classification of locally compact connected
fields \cite{pontrjagin},
Kolmogorov \cite{kolmogoroff} classified compact connected
desarguesian projective spaces, showing that the defining field
is $\mathbb R,\mathbb C$ or the real quaternion division algebra
$\mathbb H$. 
Half a century later Salzmann
\cite{szm} classified all compact connected projective planes
admitting a flag transitive, or equivalently, chamber transitive,
group of continuous automorphisms. He proved that such a projective plane is the
Moufang plane over $\mathbb R,\mathbb C,\mathbb H$ or the real octonion division algebra
$\mathbb O$. These results cover the compact connected buildings of type $\mathsf A_n$.

Burns and Spatzier \cite{burns} classified all irreducible
compact connected buildings of rank at least~2 under the condition that they
admit a strongly
transitive group of continuous automorphisms; they call this 
the `topological Moufang condition'. In particular, they
classified all compact connected Moufang buildings of rank at least~2.
Their proof does not use the classification
of Moufang buildings. Instead, they use a characterization of Furstenberg
boundaries of simple Lie groups. They also use, as does Salzmann,
the solution of the 5th Hilbert problem.

Continuing in the vein of Salzmann's result, Grundh\"ofer, Knarr and Kramer
\cite{knarr,knarr2} classified all compact connected buildings
of rank at least~2 admitting a chamber transitive group of
continuous automorphisms. As we have seen, it suffices to deal with
the rank~2 case. The final result,
assembled from \cite{burns}, \cite{knarr,knarr2} and \cite[Ch.~7]{kramerhabil},
is as follows.

\begin{theorem}\label{ConnectedClass}
Let $\Delta$ be an irreducible compact and connected building
of rank ${\it l}$ at least~$2$. Suppose that the topological automorphism group
$\mathrm{Auttop}(\Delta)$ acts transitively on the chambers of $\Delta$.
Then $\Delta$ is the standard spherical building associated to a
noncompact centerless real or complex simple Lie group $G$ of rank~${\it l}$;
cf.~Proposition~{\rm \ref{pop3}}.

\end{theorem}
We remark that such a Lie group $G$ is simple as an abstract group
\cite[94.21]{blauesbuch}. In particular, $G=G^\dagger$ in the notation
of \ref{abc45}. If $G$ is not complex, then every abstract automorphism
of $G=G^\dagger$ is automatically continuous \cite{freudenthal}.
This shows that every abstract automorphism of $\Delta$ is also
continuous. If $G$ is complex, then the automorphism group of $G=G^\dagger$
is the group of all continuous automorphisms, extended by the field
automorphisms of $\mathbb C$; see \cite{borel-tits}.
Thus, an abstract automorphism of $G$ or $\Delta$ need not be continuous
in this situation.
Nevertheless, we have the following result.

\begin{proposition}\label{uniquenessconnectedcase}
Let $\Delta$ be a compact connected irreducible Moufang building of rank
at least~$2$. If the associated simple Lie group $G$ is not complex, or equivalently,
absolutely simple as a real algebraic group,
then there is a unique topology on $\Delta$ which turns
$\Delta$ into a compact building. If $G$ is complex, then all topologies that
turn $\Delta$ into a compact building are conjugate under ${\rm Aut}(\mathbb C)$.
\end{proposition}
\begin{proof}
First of all we note that the defining field $F$ of $\Delta$ is
$\mathbb R,\mathbb C$, the real quaternion division algebra $\mathbb H$ or 
the real Cayley division algebra $\mathbb O$.
Suppose that $T$ is a compact topology on the chambers of $\Delta$ that turns $\Delta$
into a compact building. If the set of chambers $C$ is connected in this
topology, then $\Delta$ comes from a simple
Lie group $H$ by Theorem~\ref{ConnectedClass}. Therefore there is an abstract
isomorphism $H=H^\dagger\cong G^\dagger=G$. If $G$ is real, this isomorphism
is continuous by \cite{freudenthal} and if $G$ is complex, then the
isomorphism becomes continuous after composing it with a field automorphism
by \cite[8.1]{borel-tits}; see also 
\cite[5.8--5.9]{tits-spherical} for a more detailed statement.

Next, assume that the topology $T$ on $C$ is not connected. Then it is totally
disconnected by Proposition~\ref{pop2} and $\Delta$ is one of the buildings classified in the 
present paper. Therefore the center of the defining field $F$ is
a finite extension of the $p$-adic field $\mathbb Q_p$ for some prime $p$.
Neither $\mathbb R$ nor $\mathbb C$ is, however, isomorphic to
a finite extension of $\mathbb Q_p$ by, for example, \cite[54.2]{salzmann}.
\end{proof}
We remark that there is another approach to the proof of 
Proposition~\ref{uniquenessconnectedcase}:
it can be shown that a simple real or complex Lie group admits only
one locally compact and $\sigma$-compact group topology up to conjugation
by an abstract automorphism if $G$ happens to be complex; see 
\cite{kramergroups}.

\medskip
\noindent
{\em Proof of Theorem} \ref{thm2}.
The special automorphism group of a reducible building factors as a product
of the automorphism groups of the factors. In the compact case, this
factorization is compatible with the topology because the elementwise stabilizer
of any residue
is closed in $\mathrm{Auttop}(\Delta)$. In this way the problem can be reduced
to the irreducible case. Now the result follows from a combination of 
Theorems~\ref{thm1} and~\ref{ConnectedClass} and Proposition~\ref{uniquenessconnectedcase}.
\qed


\bigskip
\begin{thebibliography}{99}
\bibitem{abramenko}
P. Abramenko\ and\ K. S. Brown, {\it Buildings, theory and application}, Springer, New York, 2008.
\bibitem{borel-tits} A. Borel and J. Tits, 
Homomorphismes ``abstraits'' de groupes alg\'ebriques simples, 
{\it Ann. Math.} {\bf 97} (1973), 499--571.
\bibitem{bridson} M. R. Bridson and A. Haefliger, {\it Metric spaces of non-positive curvature},
Springer, Berlin, Heidelberg, New York, 1999.
\bibitem{brown} K. S. Brown, {\it Buildings}, Springer, Berlin, New York, Heidelberg, 1998. 
\bibitem{bruck} R. H. Bruck and E. Kleinfeld, The structure of alternative division 
rings, {\it Proc. Amer. Math. Soc.} {\bf2} (1951), 878--890.
\bibitem{bruh} F. Bruhat and J. Tits, Groupes alg\'ebriques simples sur un corps
local, I. Donn\'ees radicielles valu\'ees, {\it Inst. Hautes \'Etudes Sci. Publ. Math.} {\bf41} (1972),
5--251.
\bibitem{burns} K. Burns and R. Spatzier, On topological Tits buildings and their
classification, {\it Inst. Hautes \'Etudes Sci. Publ. Math.} {\bf 65} (1987), 5--34.
\bibitem{dug} J. Dugundji, {\it Topology}, Allyn \&\ Bacon, Boston, 1966. 
\bibitem{eberlein} P. B. Eberlein, {\it Geometry of nonpositively curved manifolds},
University of Chicago Press, Chicago, 1996.
\bibitem{eng} R. Engelking, {\it General topology}, 2nd edition, Heldermann, Berlin, 1989. 
\bibitem{freudenthal} H. Freudenthal, Die Topologie der Lieschen Gruppen 
als algebraisches Ph\"anomen. I, {\it Ann. of Math. (2)} {\bf 42} (1941), 1051--1074.
\bibitem{grund} T. Grundh\"ofer, Compact disconnected Moufang planes are Desarguesian,
{\it Arch. Math. (Basel)} {\bf49} (1987), 124--126.
\bibitem{grund1} T. Grundh\"ofer, Compact disconnected planes, inverse limits and homomorphisms, 
{\it Monatsh. Math.} {\bf105} (1988), 261--277.
\bibitem{knarr} T. Grundh\"ofer, N. Knarr and L. Kramer, Flag-homogeneous compact
connected polygons, {\it Geom. Dedicata} {\bf55} (1995), 95--114.
\bibitem{knarr2} T. Grundh\"ofer, N. Knarr\ and\ L. Kramer, Flag-homogeneous 
compact connected polygons. II, Geom. Dedicata {\bf 83} (2000), 1--29. 
\bibitem{vanmaldeghem} T. Grundh\"ofer\ and\ H. Van Maldeghem, Topological 
polygons and affine buildings of rank three, {\it Atti Sem. Mat. Fis. Univ. Modena} 
{\bf 38} (1990), 459--479. 
\bibitem{helgason} S. Helgason, {\it Differential geometry, Lie groups, and symmetric spaces}, 
Amer. Math. Soc., Providence, RI, 2001.
\bibitem{hewitt} E. Hewitt and K. A. Ross, {\it Abstract harmonic analysis} Vol. I,
Academic Press, New York, 1963.
\bibitem{hocking} J. G. Hocking\ and\ G. S. Young, {\it Topology}, Addison-Wesley 
Publishing Co., Inc., Reading, Mass., 1961.
\bibitem{humphreys} J. E. Humphreys, {\it Reflection groups and Coxeter groups}, 
Cambridge University Press, Cambridge, 1990. 
\bibitem{kleinfeld} E. Kleinfeld, Alternative division rings of characterisitic~2, 
{\it Proc. Nat. Acad. Sci. U. S.} {\bf37} (1951), 818--820
\bibitem{kolmogoroff} A. Kolmogoroff, Zur Begr\"undung der projektiven 
Geometrie, {\it Ann. of Math. (2)} {\bf 33} (1932), 175--176.
\bibitem{kramer-diss} L. Kramer, {\it Compact polygons}, Dissertation, 
Math. Fak. Univ. T\"ubingen (1994), 72~pp., arXiv:math/0104064.
\bibitem{kramerhabil} L. Kramer, Homogeneous spaces, Tits buildings, and 
isoparametric hypersurfaces, {\it Mem. Amer. Math. Soc.} {\bf 158} (2002), xvi+114pp.
\bibitem{kramergroups} L. Kramer, The topology of a simple Lie group is essentially unique,
{\it Adv. Math.}, to appear; arXiv:1009.5457.
\bibitem{petersson} H. P. Petersson, Lokal kompakte Jordan-Divisionsringe,
{\it Abh. Math. Sem. Univ. Hamburg} {\bf39} (1973), 164--179.
\bibitem{pontrjagin} L. Pontrjagin, \"Uber stetige algebraische K\"orper, 
{\it Ann. Math.} {\bf 33} (1932), 163--174.
\bibitem{ronan} M. Ronan, {\it Lectures on buildings}, University of Chicago Press,
Chicago, 2009.
\bibitem{szm} H. Salzmann, Homogene kompakte projektive Ebenen, 
{\it Pacific J. Math.} {\bf 60} (1975), 217--234.
\bibitem{blauesbuch} H. Salzmann\ et al., {\it Compact projective planes}, de Gruyter, Berlin, 1995.
\bibitem{salzmann} H. Salzmann, T. Grundh\"ofer, H. H\"ahl and R. L\"owen, 
{\it The classical fields}, Cambridge University Press, Cambridge, 2007.
\bibitem{serre} J.-P. Serre, {\it Local fields}, Springer, Berlin, New York, Heidelberg,
1979.
\bibitem{stroppel} M. Stroppel, {\it Locally compact groups}, EMS Textbk. Math., 
Europ. Math. Soc., Zurich, 2006.
\bibitem{tits-spherical} J. Tits, {\it Buildings of spherical type and finite BN-pairs}, 
Lecture Notes in Math. {\bf386}, Springer, Berlin, New York, Heidelberg, 1974. 
\bibitem{tits-local} J. Tits, Reductive groups over local fields, 
in {\it Proc. Symp. Pure Math.} {\bf33}, Part 1 
({\it Automorphic forms, representations and $L$-functions}, Corvallis
1977), 29--69, {\it Amer. Math. Soc.}, Providence, 1979.
\bibitem{como} J. Tits, Immeubles de type affine, in 
{\it Buildings and the geometry of diagrams (Como, 1984)}, 
159--190, Lecture Notes in Math. {\bf1181},
Springer, Berlin, New York, Heidelberg, 1974.
\bibitem{TW} J. Tits and R. M. Weiss, {\it Moufang polygons}, Springer, Berlin, Heidelberg,
New York, 2002.
\bibitem{wadsworth} A. R.  Wadsworth, Valuation theory on finite dimensional division 
algebras, in 
{\it Valuation theory and its applications, Vol. I (Saskatoon, 1999)}, 385--449
Field Inst. Commun. {\bf32}. Amer. Math. Soc., Providence, 2002.
\bibitem{gwarner} G. Warner, {\it Harmonic analysis on semi-simple Lie groups. I},
Springer, New York, Heidelberg, 1972.
\bibitem{warner1} S. Warner, Locally compact vector spaces and algebras over discrete
fields, {\it Trans. Amer. Math. Soc.} {\bf130} (1968), 463-493.
\bibitem{warner} S. Warner, {\it Topological fields}, North-Holland, Amsterdam, 1989.
\bibitem{weil} A. Weil, {\it Basic number theory}, Springer, Berlin, Heidelberg, New York, 1995.
\bibitem{weiss1} R. M. Weiss, {\it The structure of spherical buildings}, Princeton
University Press, Princeton, 2003.
\bibitem{weiss2} R. M. Weiss, {\it Quadrangular algebras}, Math. Notes {\bf46}, 
Princeton University Press, Princeton, 2006.
\bibitem{weiss3} R. M. Weiss, {\it The structure of affine buildings}, Ann. Math. Stud. {\bf168},
Princeton University Press, Princeton, 2009.
\end{thebibliography}
\end{document}